\documentclass[reqno, 11pt]{amsart}
\usepackage[utf8]{inputenc}
\usepackage{url}
\usepackage[colorlinks, linkcolor=blue, citecolor=blue, urlcolor=blue]{hyperref}
\usepackage{times}

\newtheorem{theorem}{Theorem}[section]
\newtheorem{lemma}[theorem]{Lemma}
\newtheorem{proposition}[theorem]{Proposition}

\theoremstyle{definition}

\newtheorem{remark}[theorem]{Remark}

\numberwithin{equation}{section}

\usepackage{bbm}
\usepackage{url}

\usepackage{euscript}

\usepackage{amsmath}
\usepackage{amsthm}
\usepackage{mathrsfs}
\usepackage{braket}
\usepackage{amsxtra}
\usepackage{amssymb}
\usepackage{pifont}
\usepackage{amsbsy}

\usepackage{graphicx}
\usepackage{epstopdf}

\usepackage{tikz}
\usepackage{tikz-cd}
\usepackage{pgfplots}
\pgfplotsset{compat=1.18}
\usetikzlibrary{arrows.meta}

\usepackage[backend=biber,style=numeric,doi=true,url=false,eprint=true]{biblatex}
\DeclareFieldFormat{title}{\mkbibemph{#1}}
\DeclareFieldFormat[article]{title}{\mkbibemph{#1}}
\DeclareFieldFormat[misc]{title}{\mkbibemph{#1}}

\DeclareFieldFormat{date}{#1}
\DeclareFieldFormat{year}{#1}
\DeclareFieldFormat[article]{issue}{#1}

\renewbibmacro*{date}{%
  \printtext[parens]{\printdate}%
}

\renewbibmacro*{issue+date}{%
  \printtext[parens]{\printdate}%
  \newunit}

\newskip\aline \newskip\halfaline
\title[Low Regularity Norm Inflation of Axisymmetric Hall MHD in Sobolev Space]{Low Regularity Norm Inflation of Axisymmetric Hall Magnetohydrodynamics in Sobolev Space}

\author{Haoming Zhu}
\address{Haoming Zhu: Department of Mathematics, Statistics, and Computer Science, University of Illinois Chicago, Chicago, IL 60607, USA}
\email{hzhu54@uic.edu}

\begin{document}

\begin{abstract}
In this paper, we study the axisymmetric Hall magnetohydrodynamics where $u = u^re_r+u^ze_z$ and $B = B^\theta e_\theta$. We construct small, smooth, compactly supported, and axisymmetric initial data $(u_{0},B_{0})$ and prove that there is a unique smooth axisymmetric solution with norm inflation in $H^{\sigma-1}(\mathbb R^3) \times H^{\sigma}(\mathbb R^3)$ for $1<\sigma<7/2$. Moreover, we study the case with fractional dissipation $(-\Delta)^\beta u$ and $(-\Delta)^\alpha B$, modifying the initial data slightly and proving that the solution inflates in $H^{\sigma-1}(\mathbb R^3) \times H^{\sigma}(\mathbb R^3)$ for $1<\sigma<7/2-2\max\{\alpha,\beta\}$ and $ 0 \le \alpha, \beta < 5/4$.
\end{abstract}

\maketitle

\section{Introduction}
We consider the inviscid Hall magnetohydrodynamics (MHD) system in $\mathbb R^3$
\begin{equation}\label{eq:hallMHD}
\begin{cases}
u_t + ( {u} \cdot \nabla) {u} - (\nabla \times  {B}) \times  {B} + \nabla p = 0\\
B_t - \nabla \times ( {u} \times  {B}) +  \nabla \times ((\nabla \times  {B}) \times  {B}) =  {0} \\
\nabla \cdot  {u} = \nabla \cdot  {B} = 0.
\end{cases}
\end{equation}
whose fractional dissipation generalization is
\begin{equation}\label{eq:hallMHD-frac}
\begin{cases}
u_t + ( {u} \cdot \nabla) {u} - (\nabla \times  {B}) \times  {B} + \nabla p+ \eta(-\Delta)^\beta  {u} = 0\\
B_t - \nabla \times ( {u} \times  {B}) +  \nabla \times ((\nabla \times  {B}) \times  {B}) + \mu(-\Delta)^\alpha  {B} =  {0} \\
\nabla \cdot  {u} = \nabla \cdot  {B} = 0.
\end{cases}
\end{equation}
Hall MHD is of interest both in physics and mathematics. In physics derivation (e.g., \cite{Acheritogaray2011} \cite{Lighthill1960}), Hall term represent a correcting term  due to the Hall effect and accounts for phenomena like magnetic reconnection (e.g., \cite{Forbes1991}). Mathematically, singularity of the Hall term $\nabla \times ((\nabla\times B)\times B)$, which has one more derivative compared to the Euler nonlinear term $u \cdot \nabla u$ thus is harder to control.

In fact, when there is no magnetic dissipation (i.e. $\mu = 0$), the Hall term indeed is a loss of derivative even if $\beta = 1$. A striking fact shown by Jeong and Oh \cite{jeongoh3Dmhd} \cite{jeongoh2.5Dmhd} is that \eqref{eq:hallMHD-frac} exhibits norm inflation of $B$ in $H^s(\mathbb R^3)$ for \textit{any} $s > 7/2$. Jeong and Oh \cite{jeongoh2.5Dmhd} also established norm inflation of $B$ in $H^s$, when $s \ge 3$, $0 \le \alpha < 1/2$, and $1 \le \beta < 3/2$, for \eqref{eq:hallMHD-frac} in space $\mathbb R^k \times \mathbb T^{3-k}$ where $k \ge 1$. If $\mu > 0$, we could get some local well posedness results given large $\alpha>0$. Chae, Wan, and Wu \cite{Chae-Wan-Wu} showed that when $\eta=0$, $\mu > 0$, and $\alpha>1/2$, \eqref{eq:hallMHD-frac} is locally well posed in $H^s(\mathbb R^d)$ for $s > 1 + d/2$. Later Dai \cite{daiLWP} showed local well posedness in $H^s(\mathbb R^d)$ for $s > 2-2\alpha + d/2$ when $\eta,\mu >0$, $\beta =1$, and $\alpha>1/2$. Due to the singularity of Hall term which competes with smoothing of dissipation, the local well posedness or ill posedness when $\alpha = 1/2$ remains an open problem, although there is a global well posedness result given small data by Zhang and Zhao \cite{ZhangZhao}. For inviscid (i.e. \eqref{eq:hallMHD}) local well posedness, which is even more difficult, it was shown by Jeong, Kim, Lee \cite{axisym-well-blow} that, roughly speaking, in the zero velocity case (i.e. electron MHD) axisymmetric \eqref{eq:hallMHD} is locally well posed for $B^\theta/r$ in $H^{s}(\mathbb R^3)$ when $s > 5/2$, but singularity forms in a finite time.

In this work, we first focus on the norm inflation of Hall MHD \eqref{eq:hallMHD}, aiming at giving a low regularity complement to Jeong and Oh \cite{jeongoh3Dmhd} in axisymmetric Hall MHD (see, e.g. \cite{axisym-well-blow}), which is a reduction of \eqref{eq:hallMHD} in the axisymmetry $u(t,r,z) = u^r(t,r,z)e_{r} + u^z(t,r,z)e_z$ and $B(t,r,z) = B^\theta(t,r,z)e_\theta$.
\begin{equation}\label{pde:axisHallMHD}
\begin{cases}
\displaystyle
\partial_t u^r
+
\left(u^r\partial_r+u^z\partial_z\right)u^r
+
\partial_r p
=
-\dfrac{(B^\theta)^2}{r},
\\[1.2em]
\displaystyle
\partial_t u^z
+
\left(u^r\partial_r+u^z\partial_z\right)u^z
+
\partial_z p
=
0,
\\[1.2em]
\displaystyle
\partial_r u^r+\dfrac{u^r}{r}+\partial_z u^z=0,
\\[1.2em]
\displaystyle
\partial_t B^\theta
+
\left(u^r\partial_r+u^z\partial_z\right)B^\theta
-
\dfrac{u^r B^\theta}{r}
=
\partial_z\left(\dfrac{(B^\theta)^2}{r}\right).
\end{cases}
\end{equation}
\begin{remark}
This is not the most general axisymmetry although it is still preserved under evolution of Hall MHD. For simplicity, let us use "axisymmetry" to denote the symmetry class $u(t,r,z) = u^r(t,r,z)e_{r} + u^z(t,r,z)e_z$ and $B(t,r,z) = B^\theta(t,r,z)e_\theta$. Also, we abuse the notation $p$ for pressure: to be precise, $p=p_{\mathrm{reduced}} = p_{\mathrm{3D}} + \frac{|B|^2}{2}$ in \eqref{pde:axisHallMHD}.
\end{remark}

Our first theorem is as follows.

\begin{theorem}\label{thm:axis}
For any $1<\sigma<7/2$ and small $\epsilon > 0$, $T_{*} > 0$, we can find small time $t_{*}< T_*$ and axisymmetric initial data $(u_{0},B_{0})$ in $C_{c}^\infty (\mathbb R^3)$ such that
\[
\lVert u_{0} \rVert_{H^{\sigma-1}} + \lVert B_{0} \rVert_{H^{\sigma}} \lesssim \epsilon.   
\]
There exists a unique smooth axisymmetric solution $(u,B)$ to \eqref{pde:axisHallMHD} on $[0,t_{*}]$ such that
\[
\lVert u(t) \rVert_{H^{\sigma-1}} \lesssim  \epsilon, \quad \forall t \in [0,t_{*}],
\]
and
\[
\lVert B(t_{*}) \rVert_{\dot{H}^\sigma} \gtrsim \epsilon^{1-\sigma}.
\]
\end{theorem}
\begin{remark}
It is interesting to point out that the proof of the theorem relies on the drift of magnetic field by velocity, but the velocity itself looks small. The reason for this is about the choice of norms: $\|u\|_{H^{\sigma-1}}$ is small but $\|\nabla u\|_{L^\infty}$ is huge (see Section \ref{sec-data-app}). Drift velocity without a uniform Lipschitz bound leads to regularity loss; see, e.g. the classical
works \cite{Diperna1989OrdinaryDE} and \cite{Diperna1987OscillationsAC}.
\end{remark}
\begin{remark}
Theorem \ref{thm:axis} complements \cite{jeongoh3Dmhd} (inflation of \eqref{eq:hallMHD} when $\sigma > 7/2$) in the low regularity, but the threshold case $\sigma = 7/2$ is still open. Also, Theorem \ref{thm:axis} is complementary to the local well posedness and blowing-up results in \cite{axisym-well-blow}.
\end{remark}

If a solution to \eqref{pde:axisHallMHD} inflates, then it also inflates as a solution to \eqref{eq:hallMHD}. However, note that the uniqueness in Theorem \ref{thm:axis} means uniqueness for the axisymmetric system \eqref{pde:axisHallMHD} instead of the full Hall MHD \eqref{eq:hallMHD}. If the solutions to \eqref{eq:hallMHD} are not \textit{unique}, then there may exist a solution to \eqref{eq:hallMHD} starting with axisymmetric initial data but becoming non-axisymmetric later. In fact, the local well posedness of \eqref{eq:hallMHD} seems not plausible due to \cite{jeongoh3Dmhd}, thus norm inflation of \eqref{pde:axisHallMHD} may not be easily inherited by 3D Hall MHD \eqref{eq:hallMHD}, although non-uniqueness itself is a kind of ill posedness.

In addition, we consider Hall MHD with fractional dissipation \eqref{eq:hallMHD-frac} under axisymmetry. That is,

\begin{equation}\label{pde:axisHallMHD-frac}
\begin{cases}
\displaystyle
\partial_t u^r
+
\left(u^r\partial_r+u^z\partial_z\right)u^r
+
\partial_r p
+
\eta \mathcal A_1^\beta u^r
=
-\dfrac{(B^\theta)^2}{r},
\\[1.2em]
\displaystyle
\partial_t u^z
+
\left(u^r\partial_r+u^z\partial_z\right)u^z
+
\partial_z p
+
\eta \mathcal A_0^\beta u^z
=
0,
\\[1.2em]
\displaystyle
\partial_r u^r+\dfrac{u^r}{r}+\partial_z u^z=0,
\\[1.2em]
\displaystyle
\partial_t B^\theta
+
\left(u^r\partial_r+u^z\partial_z\right)B^\theta
-
\dfrac{u^r B^\theta}{r}
+
\mu \mathcal A_1^\alpha B^\theta
=
\partial_z\left(\dfrac{(B^\theta)^2}{r}\right).
\end{cases}
\end{equation}
where
\[
\mathcal A_1^\beta u^r := e_r\cdot (-\Delta)^\beta(u^r e_r), \qquad \mathcal A_0^\beta u^z := e_z\cdot (-\Delta)^\beta(u^z e_z),
\]
and
\[
\mathcal A_1^\alpha B^\theta := e_\theta\cdot(-\Delta)^\alpha(B^\theta e_\theta).
\]
When $\alpha =\beta =1$,
\[
\mathcal A_0 = -\left(\partial_r^2+\frac{1}{r}\partial_r+\partial_z^2 \right), \qquad \mathcal A_1 = - \left( \partial_r^2+\frac{1}{r}\partial_r+\partial_z^2-\frac{1}{r^2} \right).
\]

\begin{theorem}\label{thm:frac}
For any $1<\sigma<7/2-2 \max \{\alpha,\beta\}$, $0 \le \alpha,\beta < 5/4$, $\mu,\eta \ge0$, and small $\epsilon > 0$, $T_{*} > 0$, we can find small time $t_{*}< T_*$ and axisymmetric initial data $(u_{0},B_{0})$ in $C_{c}^\infty (\mathbb R^3)$ such that
\[
\lVert u_{0} \rVert_{H^{\sigma-1}} + \lVert B_{0} \rVert_{H^{\sigma}} \lesssim \epsilon.   
\]
There exists a unique smooth axisymmetric solution $(u,B)$ to \eqref{pde:axisHallMHD-frac} on $[0,t_{*}]$ such that
\[
\lVert u(t) \rVert_{H^{\sigma-1}} \lesssim  \epsilon, \quad \forall t \in [0,t_{*}],
\]
and
\[
\lVert B(t_{*}) \rVert_{\dot{H}^\sigma} \gtrsim \epsilon^{1-\sigma}.
\]
\end{theorem}
\begin{remark}
Theorem \ref{thm:axis} may be viewed as a special case of Theorem \ref{thm:frac} by taking $\eta,\mu,\alpha,\beta = 0$.
\end{remark}
\begin{remark}
As discussed above, we do not have local well posedness for non-resistive 3D Hall MHD thus Theorem \ref{thm:axis} does not imply 3D inflation. However, when $\alpha > 1/2$ and $\mu > 0$, we have local well posedness for 3D Hall MHD like \cite{Chae-Wan-Wu}\footnote{We should mention that, \cite{Chae-Wan-Wu} deals with local well posedness when the velocity dissipation is zero. However, it is not hard to generalize their proof to the case under fractional velocity dissipation.}. In this case, uniqueness of solution to 3D Hall MHD with fractional dissipation \eqref{eq:hallMHD-frac} is guaranteed and every solution with axisymmetric data must also be axisymmetric. Hence, when $\alpha>1/2$ and $\mu > 0$, the solution to \eqref{eq:hallMHD-frac} with initial data $(u_{0},B_{0})$ inflates by Theorem \ref{thm:frac}.
\end{remark}

In the end, let us outline the structure of the paper. We will first prove Theorem \ref{thm:axis}, which is easier than the case with dissipation and serves as a nice toy model. In Section \ref{sec-data-app}, we will construct an axisymmetric small initial data and an axisymmetric approximate solution whose norm inflates. Then in Section \ref{sec-axi-error}, we show the difference between real and approximate solution is controlled. Lastly, we prove the generalization of fractional dissipation in Section \ref{sec-frac}, that is, Theorem \ref{thm:frac}.

\subsection*{Acknowledgment}
The author is partially supported by the National Science Foundation grant DMS-2308208. He would like to thank Mimi Dai, Xiaotong (Dawson) Yang, and ChatGPT (OpenAI) for proofreading, catching mistakes, and useful discussions. However, this is not an AI-generated work, and the author assumes full responsibility for the correctness of the work.

\section{Preliminaries}
\subsection{Notations}
We may write $f\lesssim g$ if $f \le C g$ up to a constant $C$. The case $f \gtrsim g$ is similar. If $g \lesssim f \lesssim g$, then we will write $f \simeq g$. We will also write $f \ll 1$ if $f \le C F(\lambda)$ where $F(\lambda)>0$ depends on an unchosen large parameter $\lambda>0$ and $\lim_{\lambda \to \infty}F(\lambda) = 0$. 

We denote the inhomogeneous and homogeneous Sobolev space in $\mathbb R^3$ by $W^{s,p} = W^{s,p}(\mathbb R^3)$ and $\dot W^{s,p} = \dot W^{s,p}(\mathbb R^3)$ respectively. For the inhomogeneous and homogeneous Sobolev space on the $r$-$z$ plane, we use the notations $W_{r,z}^{s,p} = W^{s,p}(\mathbb R_{r,z}^2)$ and $\dot W_{r,z}^{s,p} = \dot W^{s,p}(\mathbb R_{r,z}^2)$ respectively. Similarly, we denote $L^p = L^p (\mathbb R^3)$ and $L_{r,z}^p = L^p (\mathbb R_{r,z}^2)$. For Sobolev space, we also use the notation $H^s := W^{s,2}$.

We also use $L^p$ and Sobolev space on $\mathbb{R}^5$. When unstated, we write $L^p=L^p(\mathbb{R}^3)$, $\dot{W}^{s,p} = \dot{W}^{s,p}(\mathbb{R}^3)$, and $W^{s,p}=W^{s,p}(\mathbb{R}^3)$ for space on $\mathbb{R}^3$. For $\mathbb{R}^5$, we write $L^p(\mathbb{R}^5)$, $\dot{W}^{s,p}(\mathbb{R}^5)$, and $W^{s,p}(\mathbb{R}^5)$ explicitly. Also, in $\mathbb{R}^5$, we will write derivatives $D_{\mathbb{R}^5}$ explicitly while reserve the shorthand $D$ for $\mathbb{R}^3$.

\subsection{Axisymmetric Hall MHD with 5D dissipation}\label{subsec-5d}
For $B$-equation in \eqref{pde:axisHallMHD} and \eqref{pde:axisHallMHD-frac}, we need to estimate $\|B\|_{\dot H^k}=\|B^\theta e_{\theta}\|_{\dot H^k}$, which is not equivalent to $\|B^\theta\|_{\dot{H}^k}$ in general (But if there is no dissipation, they are equivalent thanks to compact support away from $z$-axis; see Lemma \ref{lem-bdd-supp}). In fact, as a toy case, set $a = \frac{B^\theta}{r}$, note
\begin{align*}
|\nabla (B^\theta e_{\theta})|^2 = \lvert \partial_{r}B^\theta \rvert^2 + \lvert \partial_{z}B^\theta \rvert^2 + \left\lvert  \frac{B^\theta}{r}  \right\rvert^2
&= (a+ r\partial_{r}a)^2 + r^2(\partial_{z}a)^2 + a^2 \\
&= 2a^2 + 2ra\partial_{r}a + r^2 (\partial_{r}a)^2 + r^2(\partial_{z}a)^2,
\end{align*}
and
\begin{equation}\label{eq:zako4}
\begin{split}
\lVert \nabla (B^\theta e_{\theta}) \rVert_{L^2}^2 &= 2\pi \int_{\mathbb{R}}\int_{0}^\infty [2a^2 + 2ra\partial_{r}a + r^2 (\partial_{r}a)^2 + r^2(\partial_{z}a)^2] r dr dz\\
&= 2\pi\int_{\mathbb{R}}\int_{0}^\infty [ (\partial_{r}a)^2 + (\partial_{z}a)^2 ]r^3drd z + 2\pi\int_{\mathbb{R}}\int_{0}^\infty 2ra^2 + r^2 \partial_{r}a^2 dr dz \\
& = 2\pi\int_{\mathbb{R}}\int_{0}^\infty [ (\partial_{r}a)^2 + (\partial_{z}a)^2 ]r^3drd z.
\end{split}
\end{equation}
This suggests that we consider the Sobolev norm of $a$ in $r^3 drdz$ measure, in other words, in a 5D space for axisymmetric $a$. Rewrite \eqref{pde:axisHallMHD-frac} as
\begin{equation}\label{pde:ua-frac}
\begin{split}
&\partial_t u + u\cdot\nabla u + \nabla p + \eta(-\Delta_{\mathbb{R}^3})^\beta u = -r a^2 e_r,
\\
&\partial_t a + (u-2a e_z)\cdot\nabla a + \mu\mathcal G_5^\alpha a = 0,
\\
&\nabla\cdot u=0.
\end{split}
\end{equation}
where when $\alpha = 1$, 
\[
\mathcal G_{5} := -\left( \partial_{r}^2 + \frac{3}{r}\partial_{r} +\partial_{z}^2 \right).
\]
Note $\mathcal G_5$ could be viewed as a 5D Laplacian $(-\Delta_{\mathbb{R}^5})$ acting on axisymmetric functions if $r = |y|$ for $y \in \mathbb R^4$. Thus we could generalize it fractionally using fractional Laplacian in 5D: Given $f = f(r,z)$, we define $\mathcal{P} f (y,z) = f(|y|,z)$ for $y \in \mathbb R^4$. Note $\mathcal{P}$ is a bijection from $(r,z)$-functions to axisymmetric function on $\mathbb R^5$. Then we set,
\[
\mathcal G_5^\alpha :=\mathcal{P}^{-1}(-\Delta_{\mathbb R^5})^\alpha\mathcal{P} .
\]
This 5D Laplacian construction for an axisymmetric fluid is not new; see, e.g., \cite{KochNadirashviliSereginSverak}\cite{hou2009partial}\cite{wutao-axis}. We may conclude the definition of $\mathcal G_5^\alpha$ using the following diagram.
\[
\begin{tikzcd}[column sep=large,row sep=large]
f(r,z)
\arrow[r,maps to, "\mathcal P"]
\arrow[d,maps to,"\mathcal G_5^\alpha"']
&
f(|y|,z)
\arrow[d,maps to,"(-\Delta_{\mathbb R^5})^\alpha"]
\\
\mathcal G_5^\alpha f(r,z)
\arrow[r,maps to,"\mathcal P"']
&
(-\Delta_{\mathbb R^5})^\alpha f(|y|,z).
\end{tikzcd}
\]
We also need to justify the definition of $\mathcal{G}_5^\alpha$ by checking it is indeed equivalent to the dissipation of $B^\theta$ in \eqref{pde:axisHallMHD-frac}. Although the 5D Laplacian framework is already studied in many works, for completeness we prove a lemma for a quick justification.

\begin{lemma}\label{lem-G-commute}
For $\alpha \ge 0$, axisymmetric smooth $f$ such that $\mathcal{P}f \in H^{2\alpha}(\mathbb{R}^5)$, and $e_{i} = e_\theta$ or $e_{r}$,
\[
r (\mathcal{G}_{5}^\alpha f) e_{i} = (-\Delta)^\alpha (rfe_{i}).
\]
In particular, by taking $e_{i} = e_\theta$, $(-\Delta)^\alpha (B^\theta e_{\theta}) = r(\mathcal{G}_5^\alpha a) e_\theta$.
\end{lemma}
\begin{proof}
Define
\[
\mathcal T_{i}f(r,z) := rf(r,z)e_i.
\]
Equivalently,
\[
\mathcal T_{\theta}f(r,z) = (-x_{2}f,x_{1}f,0), \qquad \mathcal T_{r}f(r,z) = (x_{1}f,x_{2}f,0).
\]
When $\alpha = 1$, we consider $-\Delta = -\left( \partial_{r}^2 + \frac{1}{r}\partial_{r} + \partial_{z}^2 \right)$ and $\mathcal G_{5} = -\left( \partial_{r}^2 + \frac{3}{r}\partial_{r} +\partial_{z}^2 \right)$. Note
\[
-\Delta (x_{j}f) = x_{j} (-\Delta)f - 2\partial_{x_{j}}f
\]
Thus, we could compute directly that
\[
(-\Delta) (x_{j}f)= x_{j}\mathcal{G}_{5}f,
\]
and the $\alpha=1$ case follows.

For the non-integer case, we set $U_{i} = \sqrt{ \pi }\mathcal{T}_i\mathcal{P}^{-1}: X  \to Y$, where $X = \{g \in L^2(\mathbb{R}^5): g (y,z)=g(|y|,z)\}$ is a closed subspace of $L^2(\mathbb{R}^5)$, and $Y = \{F \in L^2(\mathbb{R}^3;\mathbb{R}^3): F = f(r,z)e_{i}\}$ is a closed subspace of $L^2(\mathbb{R}^3;\mathbb{R}^3)$. Note
\[
\lVert U_{i} \mathcal{P}f \rVert_{L^2(\mathbb{R}^3)}^2 = 2\pi^2 \int_{\mathbb{R}}\int_{0}^\infty |f(r,z)|^2 r^3 drdz = \lVert \mathcal{P}f \rVert_{L^2(\mathbb{R}^5)}^2, 
\]
thus $U_{i}$ is unitary. Also, $(-\Delta)|_{Y}$ and $(-\Delta_{\mathbb{R}^5})|_{X}$ are self adjoint operators on $Y$ and $X$ respectively. We know $(-\Delta)|_{Y} U_i = U_i (-\Delta_{\mathbb{R}^5})|_{X}$ for smooth axisymmetric functions, which form a core to $(-\Delta_{\mathbb{R}^5})|_{X}$. Also their image under $U_{i}$ form a core to $(-\Delta)|_{Y}$. Thus taking closure we have
\[
(-\Delta)|_{Y}  = U_{i} (-\Delta_{\mathbb{R}^5})|_{X} U_{i}^{-1} = \mathcal{T}_{i} \mathcal{G}_{5} \mathcal{T}_{i}^{-1}.
\]
Recall for every Borel function $\phi$, two self adjoint operators $A$, $B$, and unitary transform $U$, we have
\[
A = UBU^{-1} \implies \phi(A) = U\phi(B)U^{-1}.
\]
See, e.g., \cite{Teschl2014}. Apply this to Borel function $\phi(s)=s^\alpha$, for $\alpha \ge 0$
\[
(-\Delta)|_{Y}^\alpha = U_{i} (-\Delta_{\mathbb{R}^5})|_{X}^\alpha U_{i}^{-1}= \mathcal{T}_{i} \mathcal{G}_{5}^\alpha \mathcal{T}_{i}^{-1},
\]
which is equivalent to what we want to show.
\end{proof}

\begin{lemma}\label{lem-Hk-3-5}
For real number $k  \ge 0$, axisymmetric smooth $f$ and $g$ such that $\mathcal{P}f,\mathcal{P}g \in H^k(\mathbb{R}^5)$, and $e_{i} = e_\theta$ or $e_{r}$,
\[
\braket{ rfe_{i} , rge_{i} }_{L^2} =  \frac{1}{\pi} \braket{ \mathcal{P}f , \mathcal{P}g  }_{L^2(\mathbb{R}^5)}  
\]
Moreover,
\[
\lVert r f e_{i} \rVert_{\dot{H}^k}  = \frac{1}{\sqrt{ \pi }} \lVert \mathcal{P} f \rVert_{\dot{H}^k(\mathbb{R}^5)}.
\]
\end{lemma}

\begin{proof}
The first part is not hard
\[
\braket{ rfe_{i} , rge_{i} }_{L^2} = |\mathbb{S}^1|\int_{\mathbb{R}}\int_{0}^\infty  fg r^3 d r dz = \frac{|\mathbb{S}^1|}{|\mathbb S^3|}\braket{ \mathcal{P}f , \mathcal{P}g  }_{L^2(\mathbb{R}^5)} = \frac{1}{\pi}\braket{ \mathcal{P}f , \mathcal{P}g  }_{L^2(\mathbb{R}^5)}.
\]
For the second part, from the definition of $\mathcal{G}_5$ and Lemma \ref{lem-G-commute},
\begin{align*}
\lVert rfe_{i} \rVert_{\dot{H}^k}^2 = \braket{ (-\Delta)^{k/2} (rf e_{i}) , (-\Delta)^{k/2} (rfe_{i}) }_{L^2} &= \braket{ r(\mathcal{G}_{5}^{k/2} f)e_{i} , r(\mathcal{G}_{5}^{k/2} f)e_{i} }_{L^2} \\
&= \frac{1}{\pi}\braket{\mathcal{P} \mathcal{G}_{5}^{k/2} f , \mathcal{P} \mathcal{G}_{5}^{k/2} f }_{L^2(\mathbb{R}^5)} \\
&= \frac{1}{\pi}\braket{(-\Delta_{\mathbb{R}^5})^{k/2} \mathcal{P} f, (-\Delta_{\mathbb{R}^5})^{k/2} \mathcal{P} f }_{L^2(\mathbb{R}^5)} \\
&= \frac{1}{\pi} \lVert \mathcal{P}f \rVert_{\dot{H}^k(\mathbb{R}^5)}^2.
\end{align*}
\end{proof}
For a vector $v(r,z)=v^r(r,z)e_r + v^z(r,z)e_z$, we define $\mathcal{P}v(y,z):=v^r(|y|,z) \frac{y}{|y|} + v^z(|y|,z)e_z$. Check that for two vectors $v_{1}=v_1(r,z)$ and $v_{2}=v_2(r,z)$,
\[
\mathcal{P}(v_{1}\cdot v_{2}) = \mathcal{P}(v_{1}^rv_{2}^r+v_{1}^zv_{2}^z) = \mathcal{P}v_{1}^r\mathcal{P}v_{2}^r + \mathcal{P}v_{1}^z\mathcal{P}v_{2}^z = \mathcal{P}v_{1} \cdot \mathcal{P}v_{2}.
\]
For scalar $f=f(r,z)$, $\nabla = \nabla_{\mathbb{R}^3} = e_r \partial_r + e_z \partial_z$, and $\nabla_{\mathbb{R}^5} = \frac{y}{|y|}\partial_r + e_z \partial_z$, we note $\mathcal{P}$ commutes with gradient
\[
\mathcal{P} \nabla f(y,z) = \partial_{r}f (|y|,z) \frac{y}{|y|} + \partial_{z}f(|y|,z)e_{z} = \nabla_{\mathbb{R}^5}\mathcal{P}f(y,z).
\]
We also have the following pointwise comparison of absolute values between 3D and 5D.
\begin{lemma}\label{lem-Looo-3-5}
For integer $k \ge 0$, axisymmetric smooth function $f$ and vector $v=v(r,z)=v^r(r,z)e_r + v^z(r,z)e_z$, then
\[
|\nabla^k f| \simeq_k |\nabla^k_{\mathbb{R}^5}\mathcal{P}f|, \qquad |\nabla^k v| \simeq_k |\nabla^k_{\mathbb{R}^5}\mathcal{P}v|.
\]
\end{lemma}
\begin{proof}
The case when $k=0$ is immediate. When $k = 1$, 
\[
|\nabla_{\mathbb{R}^3} f (r,z)| = |(\partial_{r}f) e_{r} + (\partial_{z}f) e_{z}| = |\nabla_{\mathbb{R}^5} \mathcal{P}f (r,z)|.
\]
When $k = 2$, the absolute value of Hessian on $\mathbb{R}^n$
\[
\lvert \nabla^2_{\mathbb{R}^n}f \rvert^2 = |\partial_{rr}f|^2 + (n-2)\left\lvert \frac{\partial_{r}f}{r}  \right\rvert^2 + 2|\partial_{rz}f|^2  + |\partial_{zz}f|^2 .
\]
Thus $|\nabla^2 f| \simeq_k |\nabla^2_{\mathbb{R}^5}\mathcal{P}f|$ and they differ only by the constant in front of $\frac{\partial_{r}}{r}$. For $k \ge 3$, the difference in dimensions produces different coefficients for the lower order terms, and we can show similar equivalence (whose constant depends on $k$).

$|\nabla^k v| \simeq_k |\nabla^k_{\mathbb{R}^5}\mathcal{P}v|$ like in the proof above, except note that $|\nabla e_{r}| \simeq |\nabla_{\mathbb{R}^5} \frac{y}{|y|}|\simeq\frac{1}{r}$ and $\nabla e_z =  \nabla_{\mathbb{R}^5} e_z=0$, which introduce the same coefficient thus derivatives in 3D and 5D are still comparable.
\end{proof}

At the end of this subsection, we prove a useful lemma for estimates later.
\begin{lemma}\label{lem-3-5-nonlinear}
For integer $k \ge 0$ and axisymmetric smooth $B = rae_\theta$ (thus $a$ is also smooth)
\[
|r\nabla_{\mathbb{R}^5}^k \mathcal{P}a| \lesssim |\nabla^k B|.
\]
Moreover, for integer $i,j \ge 0$,
\[
\int_{\mathbb{R}^5} |D_{\mathbb{R}^5}^i\mathcal{P}u D_{\mathbb{R}^5}^j \mathcal{P}a|^2 dx' \lesssim_{i,j} \lVert \nabla^iu\nabla^jB \rVert_{L^2}^2. 
\]
\end{lemma}
\begin{proof}
First of all, if $r=0$, then the left side vanishes thus the inequality holds immediately, so below we assume $r>0$. It is also not hard to see the inequality holds when $k=0$ thus we may assume $k \ge 1$ below. If we wish to show $|r\nabla_{\mathbb{R}^5}^k \mathcal{P}a| \lesssim |\nabla^k B|$ at $(r,z)$, it suffices to rotate around $z$-axis and show the inequality at cartesian coordinate $(x_{1},x_{2},z)=(r,0,z)\in \mathbb{R}^3$ for $|\nabla^k B(re_{r}+ze_{z}) |=  |\nabla^k B(r,0,z)|$. Also, it suffices to show $|r\nabla_{\mathbb{R}^5}^k \mathcal{P}a| \lesssim |\nabla^k A|$ where $A = rae_r$ is a rotation of $B$. Together with $|\nabla_{\mathbb{R}^5}^k \mathcal{P}a| \simeq_{k} |\nabla^k a|$, we shall prove the following: At $(x_{1},x_{2},z)=(r,0,z)$, for $r>0$ and $k \ge 1$,
\[
|r\nabla^k a| \lesssim_{k} |\nabla^kA| = |\nabla^k (x_{1}a,x_{2}a,0)|.
\]
Note for any $\kappa = (\kappa_{1},\kappa_{2},\kappa_{3})$ with $|\kappa|=k$
\[
D^\kappa A_{1} = x_{1}D^\kappa a + \kappa_{1}D^{\kappa-e_{1}}a = rD^\kappa a + \kappa_{1}D^{\kappa-e_{1}}a.
\]
where $D$ is the Cartesian derivative. We need to control $D^{\kappa-e_{1}}a$. For $\omega = (\omega_{1},\omega_{2},\omega_{3})$ such that $|\omega|=k-1$,
\[
D^{\omega+e_{2}}A_{2} = x_{2}D^{\omega+e_{2}}a + (\omega_{2}+1)D^\omega a=(\omega_{2}+1)D^\omega a.
\]
since $x_{2}=0$. Then
\[
|rD^\kappa a| \le |D^\kappa A_{1}| + \kappa_{1}|D^{\kappa-e_{1}}a| = |D^\kappa A_{1}| + \frac{\kappa_{1}}{1+\kappa_{2}}|D^{\kappa-e_{1}+e_{2}}A_{2}| \lesssim |\nabla^k A|.
\]
So, the first part of the lemma is proved. By the first part of the lemma and Lemma \ref{lem-Looo-3-5}
\begin{align*}
\int_{\mathbb{R}^5} |\nabla_{\mathbb{R}^5}^i\mathcal{P}u \nabla_{\mathbb{R}^5}^j \mathcal{P}a|^2 dx' &\lesssim_{i,j} \int_{\mathbb{R}}\int_{0}^\infty|\nabla^i u|^2|\frac{1}{r}\nabla^jB|^2 r^3 drd z \\
&= \int_{\mathbb{R}}\int_{0}^\infty|\nabla^i u|^2|\nabla^jB|^2 r drd z \\
&\lesssim \lVert \nabla^iu\nabla^jB \rVert_{L^2}^2.
\end{align*}
\end{proof}

\subsection{A priori estimate}
We need the existence and uniqueness of axisymmetric solution on $[0,t_{*}]$ as in the statements of Theorem \ref{thm:axis} and \ref{thm:frac}. There are many local well posedness results on \eqref{pde:axisHallMHD-frac} with various choice of $\eta,\mu,\beta,\alpha\ge 0$; see, e.g. \cite{axisym-well-blow} and \cite[Section 2.2]{jeongoh3Dmhd}, but we did not find a general well posedness result for all combinations of $\eta,\mu,\beta,\alpha \ge 0$. So, we prove a priori estimate here, which is partially covered by \cite{axisym-well-blow}\cite{jeongoh3Dmhd}.
\begin{proposition}\label{prop-LWP}
For any initial data $(u_{0},B_{0})$ to \eqref{pde:axisHallMHD-frac}, where $u_{0}(r,z)=u_0^r(r,z)e_r + u_0^z(r,z)e_z$, $\nabla \cdot u_0 =0$, and $B_{0}(r,z)=B_0^\theta(r,z)e_\theta$, such that $(u_{0},B_{0})\in H^m \times H^m$ for integer $m \ge 4$, there exists a unique solution $(u,B)\in L^\infty([0,T];H^m \times H^m)$ with $(u(0),B(0))=(u_{0},B_{0})$. 

Here fix any integer $m' \in (7/2, m]$, $T=T_{m'} = T_{m'} (m', \lVert u_{0} \rVert_{H^{m'}}, \lVert B_{0} \rVert_{H^{m'}})$ independent of higher Sobolev norms, that is, we have a uniform time interval $[0,T_{m'}]$ for all $m \ge m'$. So for $(u_{0}, B_{0})\in H^\infty \times H^\infty$, we have a unique smooth solution $(u,B)\in L^\infty([0,T_{m'}];H^\infty \times H^\infty)$.
\end{proposition}
\begin{remark}
Making full use of fractional Sobolev space and dissipation, Proposition \ref{prop-LWP} is likely true for $m > 7/2$ and, if $\eta, \mu >0$, $m > 7/2 - 2 \min\{\alpha,\beta\}$. But this is not the main purpose of the paper and let us not pursue optimization.
\end{remark}

\begin{proof}
Applying $\mathcal{P}$ to $a$-equation of \eqref{pde:ua-frac}, we lift \eqref{pde:ua-frac} to 3D-5D mixed system
\begin{equation}\label{pde-ua-frac-3-5}
\begin{split}
&\partial_t u + u\cdot\nabla u + \nabla p + \eta(-\Delta)^\beta u = -r a^2 e_r,
\\
&\partial_t \mathcal{P}a + (\mathcal{P}u-2\mathcal{P}a e_z)\cdot\nabla_{\mathbb{R}^5} \mathcal{P} a + \mu(-\Delta_{\mathbb{R}^5})^\alpha \mathcal{P}a = 0,
\\
&\nabla\cdot u=0.
\end{split}
\end{equation}
For $0 \le k \le m$, differentiate \eqref{pde-ua-frac-3-5} using $k$-th cartesian derivative $D^k$ and $D_{\mathbb{R}^5}^k$ respectively, and integrate with $D^ku$ and $D_{\mathbb{R}^5}^k\mathcal{P}a$,
\begin{align*}
& \frac{1}{2} \partial_{t} \lVert D^k u \rVert_{L^2}^2 + \eta \lVert \Lambda^{k+\beta} u \rVert_{L^2}^2  + \int_{\mathbb{R}^3} D^k(u \cdot \nabla u) D^kudx + \int_{\mathbb{R}^3}\nabla D^k p \cdot D^kudx \\
& \qquad \qquad \qquad \qquad\qquad\qquad\qquad\qquad\qquad\qquad\qquad\qquad\qquad= - \int_{\mathbb{R}^3} D^k(ra^2 e_{r}) D^k u dx,\\
&\frac{1}{2}\partial_{t} \lVert D_{\mathbb{R}^5}^k \mathcal{P}a \rVert_{L^2(\mathbb{R}^5)}^2 + \mu \lVert \Lambda_{\mathbb{R}^5}^{k+\alpha} \mathcal{P}a \rVert_{L^2(\mathbb{R}^5)}^2 + \int_{\mathbb{R}^5} D_{\mathbb{R}^5}^k  ( (\mathcal{P}u-2\mathcal{P}a e_z)\cdot\nabla_{\mathbb{R}^5} \mathcal{P}a) D_{\mathbb{R}^5}^k \mathcal{P} a dx' = 0.
\end{align*}
For the first line, using $\nabla \cdot u = 0$ as in Euler case and drop the dissipation,
\begin{align*}
\partial_{t} \lVert D^k u \rVert_{L^2}^2 \lesssim \lVert \nabla u \rVert_{L^\infty}\lVert D^ku \rVert_{L^2}^2 + \lVert D^k (ra^2 e_{r})\rVert_{L^2}\lVert D^ku \rVert_{L^2} . 
\end{align*}
where the only difference compared to Euler or Navier-Stokes is that we have a forcing term consists of $a$. Note from Lemma \ref{lem-Hk-3-5} and Leibniz,
\begin{align*}
\lVert D^k (ra^2 e_{r}) \rVert_{L^2} \simeq \lVert \mathcal{P}(a^2) \rVert_{\dot{H}^k(\mathbb{R}^5)} =& \lVert (\mathcal{P}a)^2 \rVert_{\dot{H}^k(\mathbb{R}^5)} \lesssim \lVert \mathcal{P}a \rVert_{L^\infty(\mathbb{R}^5)} \lVert \mathcal{P}a \rVert_{H^k(\mathbb{R}^5)}. 
\end{align*}
Then, 
\[
\partial_{t}\lVert D^ku \rVert_{L^2}^2 \lesssim  \lVert \nabla u \rVert_{L^\infty} \lVert D^k u \rVert_{L^2}^2  + \lVert \mathcal{P}a \rVert_{L^\infty(\mathbb{R}^5)} \lVert \mathcal{P}a \rVert_{H^k(\mathbb{R}^5)}\lVert D^ku \rVert_{L^2}.
\]
For the second line, when $k=0$
\[
\int_{\mathbb{R}^5}  ( (\mathcal{P}u-2\mathcal{P}a e_z)\cdot\nabla_{\mathbb{R}^5}\mathcal{P}a)  \mathcal{P} a dx' \lesssim (\lVert (D_{\mathbb{R}^5} \mathcal{P}u) \mathcal{P}a \rVert_{L^2(\mathbb{R}^5)}+\lVert (D_{\mathbb{R}^5} \mathcal{P}a)  \mathcal{P}a \rVert_{L^2(\mathbb{R}^5)} ) \lVert \mathcal{P}a \rVert_{L^2(\mathbb{R}^5)}
\]
and for $k \ge 1$
\begin{align*}
&\int_{\mathbb{R}^5} D_{\mathbb{R}^5}^k  ( (\mathcal{P}u-2\mathcal{P}a e_z)\cdot\nabla_{\mathbb{R}^5}\mathcal{P}a) D_{\mathbb{R}^5}^k \mathcal{P} a dx' \\
\lesssim& \lVert D_{\mathbb{R}^5}^{k} (\mathcal{P}u \cdot \nabla_{\mathbb{R}^5} \mathcal{P}a) - (\mathcal{P}u \cdot D_{\mathbb{R}^5}^{k}\nabla_{\mathbb{R}^5}\mathcal{P}a) \rVert_{L^2(\mathbb{R}^5)}\lVert D_{\mathbb{R}^5}^k \mathcal{P}a \rVert_{L^2(\mathbb{R}^5)}\\
&+\lVert D_{\mathbb{R}^5}^{k}( \mathcal{P}a \nabla_{\mathbb{R}^5}\mathcal{P}a)^2 - \mathcal{P}a D_{\mathbb{R}^5}^{k}\nabla_{\mathbb{R}^5}\mathcal{P}a \rVert_{L^2(\mathbb{R}^5)} \lVert D_{\mathbb{R}^5}^k \mathcal{P}a \rVert_{L^2(\mathbb{R}^5)} \\
&+  \lVert \nabla_{\mathbb{R}^5} \mathcal{P}a \rVert_{L^\infty} \lVert D_{\mathbb{R}^5}^k \mathcal{P}a \rVert_{L^2(\mathbb{R}^5)}^2 + \lVert \nabla u \rVert_{L^\infty} \lVert B \rVert_{\dot{H}^k}\lVert D_{\mathbb{R}^5}^k \mathcal{P}a \rVert_{L^2(\mathbb{R}^5)}  ,
\end{align*}
where we used integral by parts $\int_{\mathbb{R}^5}   ( (\mathcal{P}u-2\mathcal{P}a e_z)\cdot\nabla_{\mathbb{R}^5}D_{\mathbb{R}^5}^k\mathcal{P}a) D_{\mathbb{R}^5}^k \mathcal{P} a dx' = -\frac{1}{2}\int_{\mathbb{R}^5} \nabla_{\mathbb{R}^5} \cdot (\mathcal{P}u -2\mathcal{P}ae_{z})|D_{\mathbb{R}^5}^k\mathcal{P}a|^2 dx'$ and Lemma \ref{lem-3-5-nonlinear}. By commutator estimate and Lemma \ref{lem-3-5-nonlinear},
\begin{align*}
&\lVert D_{\mathbb{R}^5}^{k} (\mathcal{P}u \cdot \nabla_{\mathbb{R}^5} \mathcal{P}a) - (\mathcal{P}u \cdot D_{\mathbb{R}^5}^{k}\nabla_{\mathbb{R}^5}\mathcal{P}a) \rVert_{L^2(\mathbb{R}^5)} \\
&\lesssim \lVert B \rVert_{H^k} \lVert \nabla u \rVert_{L^\infty} + \lVert u \rVert_{H^k}\lVert \nabla B\rVert_{L^\infty}    \\
&\lVert D_{\mathbb{R}^5}^k( \mathcal{P}a)^2 - \mathcal{P}a D_{\mathbb{R}^5}^{k+1}\mathcal{P}a \rVert_{L^2(\mathbb{R}^5)} \lesssim \lVert \mathcal{P}a \rVert_{H^k(\mathbb{R}^5)}\lVert \nabla_{\mathbb{R}^5}\mathcal{P}a \rVert_{L^\infty(\mathbb{R}^5)} .
\end{align*}
So, dropping dissipation in the estimate of $a$,
\begin{align*}
\partial_{t}\lVert D_{\mathbb{R}^5}^k \mathcal{P}a \rVert_{L^2(\mathbb{R}^5)}^2  \lesssim  \lVert B \rVert_{H^k}  \lVert \nabla u \rVert_{L^\infty}\lVert \mathcal{P}a \rVert_{H^k(\mathbb{R}^5)} &+ \lVert u \rVert_{H^k}\lVert \nabla B\rVert_{L^\infty}\lVert \mathcal{P}a \rVert_{H^k(\mathbb{R}^5)} \\
&+ \lVert \nabla_{\mathbb{R}^5} \mathcal{P}a \rVert_{L^\infty} \lVert \mathcal{P}a \rVert_{H^k(\mathbb{R}^5)}^2.
\end{align*}
Add all the derivatives for $0 \le k\le m$, recall $\lVert \mathcal{P}a \rVert_{H^m(\mathbb{R}^5)} \simeq \|B\|_{H^m}$ from Lemma \ref{lem-Hk-3-5}
\[
\partial_{t}(\lVert u \rVert_{H^m}^2 + \lVert B \rVert_{H^m}^2  ) \lesssim (\lVert \nabla u \rVert_{L^\infty} + \lVert \mathcal{P}a \rVert_{L^\infty(\mathbb{R}^5)} + \lVert \nabla_{\mathbb{R}^5} \mathcal{P}a \rVert_{L^\infty(\mathbb{R}^5)}  )(\lVert u \rVert_{H^m}^2 + \lVert B \rVert_{H^m}^2).  
\]
For any $7/2 < m' \le m$, by Sobolev embedding,
\[
\partial_{t}(\lVert u \rVert_{H^m}^2 + \lVert B \rVert_{H^m}^2  ) \lesssim (\lVert u \rVert_{H^{m'}}  + \lVert \mathcal{P}a \rVert_{H^{m'}(\mathbb{R}^5)} )(\lVert u \rVert_{H^m}^2 + \lVert B \rVert_{H^m}^2).  
\]
Recall $\lVert \mathcal{P}a \rVert_{H^m(\mathbb{R}^5)} \simeq \|B\|_{H^m}$,
\[
\partial_{t}(\lVert u \rVert_{H^m}^2 + \lVert B \rVert_{H^m}^2  ) \lesssim (\lVert u \rVert_{H^{m'}}  + \lVert B \rVert_{H^{m'}})(\lVert u \rVert_{H^m}^2 + \lVert B \rVert_{H^m}^2).  
\]
When $m=m'$,
\[
\partial_{t}(\lVert u \rVert_{H^{m'}}^2 + \lVert B \rVert_{H^{m'}}^2  ) \lesssim (\lVert u \rVert_{H^{m'}}^2 + \lVert B \rVert_{H^{m'}}^2)^{3/2}.  
\]
thus
\[
\sup_{t \in [0,T_{m'}]}(\lVert u \rVert_{H^{m'}}^2 + \lVert B \rVert_{H^{m'}}^2  ) \lesssim \lVert u_{0} \rVert_{H^{m'}}^2 + \lVert B_{0} \rVert_{H^{m'}}^2, \quad T_{m'} = T_{m'} (m', \lVert u_{0} \rVert_{H^{m'}}, \lVert B_{0} \rVert_{H^{m'}}) .
\]
For any $m \ge m'$ and $t \in [0,T_{m'}]$, by Gr\"onwall estimate,
\begin{equation}\label{eq:priori}
\begin{split}
\lVert u \rVert_{H^m}^2 + \lVert B \rVert_{H^m}^2 &\lesssim (\lVert u_{0} \rVert_{H^m}^2 + \lVert B_{0} \rVert_{H^m}^2 ) e^{C\int_{0}^t(\lVert u \rVert_{H^{m'}}  + \lVert B \rVert_{H^{m'}})(\tau)d\tau} \\
&\lesssim (\lVert u_{0} \rVert_{H^m}^2 + \lVert B_{0} \rVert_{H^m}^2 )e^{T_{m'}(\lVert u_{0} \rVert_{H^{m'}} + \lVert B_{0} \rVert_{H^{m'}}  )},
\end{split}
\end{equation}

The uniqueness of $(u,B)$ follows by doing a similar priori estimate  to the difference as above. For existence, we could use the Friedrichs method. Define a Fourier truncation $J_N^d$ in $\mathbb{R}^d$ such that for $f$ on $\mathbb{R}^d$,
\[
\widehat{J_{N}^df} := \chi_{|\xi|<N}\hat{f}(\xi)
\]
and the space
\[
X^m_{N} := \{(u,f): u \in H^m(\mathbb{R}^3), f \in H^m(\mathbb{R}^5), J_{N}^3 u = u, J_{N}^5 f = f, \nabla \cdot u = 0. \}. 
\]
Consider the truncated \eqref{pde-ua-frac-3-5} system
\begin{align*}
&\partial_{t}u_{N} + \eta (-\Delta)^\beta u_{N} = -J_{N}^3 \mathbb{P} (u_{N}\cdot \nabla u_{N} + r a_{N}^2e_{r}),\\
&\partial_{t} \mathcal{P}a_{N} + \mu (-\Delta_{\mathbb{R}^5})^\alpha \mathcal{P}a_{N} = - J_{N}^5 \left( ( \mathcal{P}u_{N} - 2\mathcal{P}a_{N}e_{z} )\cdot \nabla_{\mathbb{R}^5} \mathcal{P}a_{N} \right).
\end{align*}
where $\mathbb{P}$ is the 3D Leray projection. For each fixed $N$, we could verify the right side is locally Lipschitz on $X_N^m$. For instance,
\begin{align*}
\lVert J_{N}^5(\mathcal{P}u_{N} \cdot \nabla_{\mathbb{R}^5}\mathcal{P}a_{N}) \rVert_{H^m(\mathbb{R}^5)}  &\lesssim (1+N^2)^{m/2} \lVert \mathcal{P}u_{N} \rVert_{L^\infty(\mathbb{R}^5)}\lVert \nabla_{\mathbb{R}^5}\mathcal{P}a_{N} \rVert_{L^2(\mathbb{R}^5)}  \\
&\simeq (1+N^2)^{m/2} \lVert u_{N} \rVert_{L^\infty}\lVert \nabla_{\mathbb{R}^5}\mathcal{P}a_{N} \rVert_{L^2(\mathbb{R}^5)}  \\
&\lesssim_{N,m} \lVert u_{N} \rVert_{H^m}\lVert \mathcal{P}a_{N} \rVert_{H^m(\mathbb{R}^5)}
\end{align*}
and it is similar for the difference $w_{N}=u_N - v_N$, $h_{N}=a_N-b_N$. So, by Picard's theorem, we have a unique local solution in $X^m_N$. Note $J_{N}^d$ is self adjoint, commutes with derivatives, and $\|u_{0,N}\|_{H^m}^2+\|\mathcal{P}a_{0,N}\|_{H^m(\mathbb{R}^5)}^2 \le \|u_{0}\|_{H^m}^2+\|\mathcal{P}a_{0}\|_{H^m(\mathbb{R}^5)}^2$, the same priori estimate \eqref{eq:priori} holds uniformly in $N$, thus by continuation of ODEs we have a common existence interval for all $N$. It is also not hard to see that $\{(u_{N},a_{N})\}$ forms a Cauchy sequence in $L^2$. For instance, by Lemma \ref{lem-3-5-nonlinear}
\[
\lVert \mathcal{P}(u_{N}-u_{M}) \cdot \nabla_{\mathbb{R}^5}\mathcal{P}a_{N} \rVert_{L^2(\mathbb{R}^5)} \lesssim \lVert u_{N} - u_{M} \rVert_{L^2}\lVert \nabla B_{N} \rVert_{L^\infty}
\]
and the difference of projection $J_{N}^5-J_M^5$ could also be controlled. Therefore, the existence of solution is proved.
\end{proof}

\section{Initial Data and Approximate Solution}\label{sec-data-app}
Before we start, we point out that our initial data set and approximate solution set is exactly in the same spirit as that of Luo \cite{luo2024}. We set large parameters $\lambda,\nu$ as follows.
\begin{equation}\label{eq:nu-lam}
\nu = \lambda^{1- \frac{7/2-\sigma}{100}} \ll \lambda.
\end{equation}
Also, for the smallness of the initial data, let us pick a small factor $\epsilon>0$. We consider the shifted polar coordinates on $(r,z)$ plane where
\[
r = \nu^{-1} + \rho\cos \varphi, \qquad z = \rho \sin \varphi.
\]
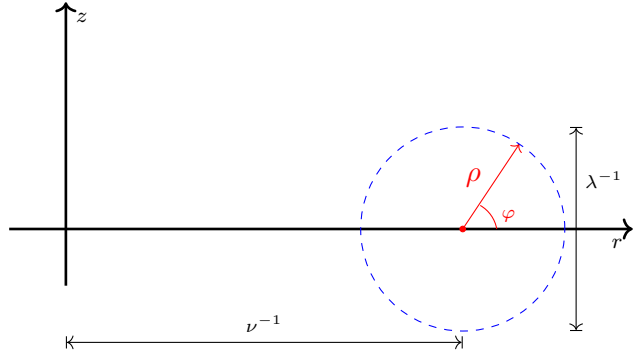
\begin{figure}[ht]
 \centering
\begin{tikzpicture}[scale=1.5]

% horizontal axis
\draw[->][line width=1]  (-0.5,-0.5) -- (-0.5,2);
\draw  (-0.5,2) node[anchor=north west] {\scriptsize  $z$};
% vertical axis
\draw[->][line width=1] (-1,0) -- (4.5,0);
 \draw (4.5,0)  node[anchor=north east] {\scriptsize $r$};

\draw[->][color= red  ] (3,0) -- (3.5,0.75);
\draw  (3.25,0.45) node[anchor= east] {  $\color{red} \rho$  };

\filldraw[color= red  ]  (3,0)  circle (0.7pt); 

 \draw[|<->|] (-0.5,-1)--++ (3.5,0) node[pos=0.5,above]{\tiny $\nu^{-1}$};

 \draw[|<->|] (4,-0.9)--++ (0,1.8); 

\draw (4.5, 0.6) node[anchor=north east]{\tiny $\lambda^{-1}$};

 \draw[blue  ,dashed] (3,0) circle (0.9cm);

\draw[color= red  ] ( 3.3,0) arc(10:60:0.3) node[pos=0.5,right]{\tiny $\varphi$};;

%%axis labels
%\filldraw 	(4,0)circle (0pt) node[anchor=south west] {\tiny $r$} ;
%\filldraw 	(0,3)circle (0pt) node[anchor=south west] {\tiny $z$} ;
\end{tikzpicture}
\caption{Schematic of the shifted polar coordinate $(\rho , \varphi)$ in $rz$-plane \cite{luo2024}. }
\label{fig:schematic}
\end{figure}
We set initial data $(u_{0},B_{0}) = (u_{0}^r e_{r} + (u_{0}^z+u_{c}^z) e_z, B_{0}^\theta e_{\theta})$ as below
\begin{equation}\label{ini_data}
\begin{split}
&u_{0}^r(\rho,\varphi) = -\epsilon \lambda^{2-\sigma}  \nu^{1/2} f'(\lambda \rho)\partial_{z}\rho,\\
&u_{0}^z(\rho,\varphi) = \epsilon \lambda^{2-\sigma}  \nu^{1/2} f'(\lambda \rho)\partial_{r}\rho,\\
&u_{c}^z(\rho,\varphi) = \epsilon \lambda^{1-\sigma}  \nu^{1/2} \frac{f(\lambda \rho)}{r},\\
&B_{0}^\theta(\rho,\varphi) = \epsilon \lambda^{1-\sigma}\nu^{1/2}g(\lambda \rho) \cos (\varphi)
\end{split}
\end{equation}
where $f,g \in C_c^\infty(\mathbb{R})$, $g \not \equiv 0$ is supported on $[2,3]$, $f \not \equiv 0$ is supported on $[1,4]$, and $f'=1$ on $[2,3]$. Note that the initial data are supported in a torus of radius with radius $\sim\nu^{-1}$, cross section's radius $\sim\lambda^{-1}$, and volume $\sim \lambda^{-2} \nu^{-1}$. It is not hard to check $(u_{0},B_{0})$ satisfies the divergence free property.
\begin{align*}
&\mathrm{div} u_{0} = \partial_{r}u_{0}^r + \frac{u_{0}^r}{r} + \partial_{z }(u_{0}^z+u_{c}^z) = 0, \\
&\mathrm{div} B_{0} = \frac{1}{r}\partial_{\theta} B_{0}^\theta =0.
\end{align*}
And we note $(u_{0}^r,u_{0}^z)$ forms a stationary solution to 2D Euler equation.
\begin{lemma} \label{lem:ini_data}
For $k \ge 0$, $1 \le p \le \infty$
\[
\lVert u_{c} \rVert_{W_{r,z}^{k,p}} \lesssim \epsilon \lambda^{k+1-\sigma} (\lambda^{-1}\nu) (\lambda^{1-2/p}\nu^{1/2}), \quad \lVert u_{c} \rVert_{W^{k,p}} \lesssim \epsilon \lambda^{k+1-\sigma} (\lambda^{-1}\nu) (\lambda^{1-2/p}\nu^{1/2-1/p}),
\]
and
\begin{align*}
&\lVert u_{0} \rVert_{W_{r,z}^{k,p}} \lesssim \epsilon \lambda^{k+1-\sigma} (\lambda^{1-2/p}\nu^{1/2}), \quad \lVert u_{0} \rVert_{W^{k,p}} \lesssim \epsilon \lambda^{k+1-\sigma} (\lambda^{1-2/p}\nu^{1/2-1/p}),\\
&\lVert B_{0} \rVert_{W_{r,z}^{k,p}} \lesssim \epsilon \lambda^{k-\sigma} (\lambda^{1-2/p}\nu^{1/2}), \quad \lVert B_{0} \rVert_{W^{k,p}} \lesssim \epsilon \lambda^{k-\sigma} (\lambda^{1-2/p}\nu^{1/2-1/p}).
\end{align*}
\end{lemma}

\begin{proof}
We may only consider the case where $k \in \mathbb N$. After that, for non-integer $k \ge 0$ an interpolation gives the same estimate. It is not hard to see on $(r,z)$ plane, with $r \sim \nu^{-1}$
\[
\lVert u_{c}^z \rVert_{L_{r,z}^\infty} \lesssim \epsilon \lambda^{1-\sigma}(\lambda^{-1}\nu) (\lambda\nu^{1/2}), \quad \lVert u_{0} \rVert_{L_{r,z}^\infty} \lesssim \epsilon \lambda^{1-\sigma} (\lambda\nu^{1/2}), \quad \lVert B_{0} \rVert_{L_{r,z}^\infty} \lesssim \epsilon \lambda^{-\sigma} (\lambda\nu^{1/2}). 
\]
and on $\mathbb R^3$
\[
\lVert u_{c}^z \rVert_{L^\infty} \lesssim \epsilon \lambda^{1-\sigma}(\lambda^{-1}\nu) (\lambda\nu^{1/2}), \quad \lVert u_{0} \rVert_{L^\infty} \lesssim \epsilon \lambda^{1-\sigma} (\lambda\nu^{1/2}), \quad \lVert B_{0} \rVert_{L^\infty} \lesssim \epsilon \lambda^{-\sigma} (\lambda\nu^{1/2}). 
\]
Note in $(r,z)$ plane $|\text{supp}_{r,z}| \sim \lambda^{-2}$ and in $\mathbb R^3$ $|\text{supp}| \sim \lambda^{-2} \nu^{-1}$. Use inequality between $L^p$ norm on finite measure, we have
\[
\begin{split}
&\lVert u_{c}^z \rVert_{L_{r,z}^p} \le |\text{supp}_{r,z}|^{1/p} \lVert u_{c}^z \rVert_{L_{r,z}^\infty} \lesssim \epsilon \lambda^{1-\sigma}(\lambda^{-1}\nu) (\lambda^{1-2/p}\nu^{1/2}), \\
&\lVert u_{0} \rVert_{L_{r,z}^p} \le  |\text{supp}_{r,z}|^{1/p} \lVert u_{0} \rVert_{L_{r,z}^\infty} \lesssim \epsilon\lambda^{1-\sigma} (\lambda^{1-2/p}\nu^{1/2}),\\
&\lVert B_{0} \rVert_{L_{r,z}^p} \le  |\text{supp}_{r,z}|^{1/p} \lVert B_{0} \rVert_{L_{r,z}^\infty} \lesssim \epsilon \lambda^{-\sigma} (\lambda^{1-2/p}\nu^{1/2}),
\end{split}
\]
and
\[
\begin{split}
&\lVert u_{c}^z \rVert_{L^p} \le |\text{supp}|^{1/p} \lVert u_{c}^z \rVert_{L^\infty} \lesssim \epsilon \lambda^{1-\sigma}(\lambda^{-1}\nu) (\lambda^{1-2/p}\nu^{1/2-1/p}), \\
&\lVert u_{0} \rVert_{L^p} \le  |\text{supp}|^{1/p} \lVert u_{0} \rVert_{L^\infty} \lesssim \epsilon\lambda^{1-\sigma} (\lambda^{1-2/p}\nu^{1/2-1/p}),\\
&\lVert B_{0} \rVert_{L^p} \le  |\text{supp}|^{1/p} \lVert B_{0} \rVert_{L^\infty} \lesssim \epsilon \lambda^{-\sigma} (\lambda^{1-2/p}\nu^{1/2-1/p}),
\end{split}
\]
Notice in $(\rho,\varphi,\theta)$ coordinate, $\nabla = e_{\rho} \partial_{\rho} + e_{\varphi} \rho^{-1}\partial_{\varphi} + e_{\theta} r^{-1}\partial_\theta \sim e_{\rho} \partial_{\rho} + e_{\varphi} \lambda\partial_{\varphi} + e_{\theta} \nu\partial_\theta$ since $r \sim \nu^{-1}$ and $\rho \sim \lambda^{-1}$. So, for $k \ge 0$,
\begin{align*}
&\lVert D^k u_{c}^z \rVert_{L_{r,z}^\infty} = \lVert D^k u_{c}^z \rVert_{L^\infty} \lesssim \epsilon \lambda^{k+1-\sigma}(\lambda^{-1}\nu) (\lambda\nu^{1/2}), \\
&\lVert D^k u_{0} \rVert_{L_{r,z}^\infty}= \lVert D^k u_{0} \rVert_{L^\infty} \lesssim \epsilon \lambda^{1+k-\sigma} (\lambda\nu^{1/2}), \\
&\lVert D^k B_{0} \rVert_{L_{r,z}^\infty}=\lVert D^k B_{0} \rVert_{L^\infty} \lesssim \epsilon \lambda^{k-\sigma} (\lambda\nu^{1/2}). 
\end{align*}
For $W^{k,p}$ norm with $1\le p < \infty$, a similar multiplication by $|\mathrm{supp}_{r,z}|^{1/p}$ or $|\mathrm{supp}|^{1/p}$ finishes the proof.
\end{proof}

Consider the solution to the approximation system
\begin{equation}\label{pde:approx}
\begin{split}
&\bar{u} \equiv u_{0},\qquad \bar{B}^r = \bar{B}^z \equiv 0\\
& \partial_{t} \bar{B}^\theta + (u_{0}^r \partial_{r} + u_{0}^z \partial_{z})\bar{B}^\theta = 0.
\end{split}
\end{equation}

\begin{lemma}\label{lem:Bbar}
For $t_{*} = \epsilon^{-2}\lambda^{\sigma-3}\nu^{-1/2} $, on $[0,t_{*}]$
\[
\lVert \bar{B}^{\theta} \rVert_{\dot{W}_{r,z}^{k,p}} \lesssim \epsilon^{1-k} \lambda^{k-\sigma}\lambda^{1-2/p}\nu^{1/2}, \quad \lVert \bar{B}^{\theta} \rVert_{\dot{W}^{k,p}} \lesssim \epsilon^{1-k} \lambda^{k-\sigma}\lambda^{1-2/p}\nu^{1/2 -1/p},
\]
and when $t = t_*$
\begin{equation*}
\lVert \bar{B}^{\theta} (t_{*}) \rVert_{\dot{H}^k_{r,z}} \simeq \epsilon^{1-k}\lambda^{k-\sigma}\nu^{1/2}, \qquad \lVert \bar{B}^{\theta} (t_{*}) \rVert_{\dot{H}^k} \simeq \epsilon^{1-k}\lambda^{k-\sigma}. 
\end{equation*}
Moreover, for $\bar{B}= \bar{B}^\theta e_\theta$
\[
\lVert \bar{B} (t_{*}) \rVert_{\dot{H}^k} \simeq \epsilon^{1-k}\lambda^{k-\sigma}. 
\]
\end{lemma}

\begin{proof}
Note
\[(u_0^r \partial_r + u_0^z \partial_z) = \omega(\rho) \frac{\partial}{\partial \varphi},\]
where $\omega(\rho) = \frac{\epsilon \lambda^{2-\sigma} \nu^{1/2} f'(\lambda \rho)}{\rho}$.
So in shifted polar coordinate, we could solve $B$ equation in \eqref{pde:approx} explicitly
\begin{equation}\label{eq:abar}
\bar{B}^\theta(t,\rho,\varphi) = B_{0}^\theta(\rho,\varphi - \omega(\rho)t) = \epsilon \lambda^{1-\sigma}\nu^{1/2}g(\lambda \rho) \cos ( \varphi - \frac{\epsilon t \lambda^{2-\sigma} \nu^{1/2} f'(\lambda \rho)}{\rho}).
\end{equation}
which is supported in a torus with radius $\sim \nu^{-1}$ and cross-section radius $\sim \lambda^{-1}$. So, $\frac{\epsilon t \lambda^{2-\sigma} \nu^{1/2} f'(\lambda \rho)}{\rho} \sim \epsilon t \lambda^{3-\sigma}\nu^{1/2} f'(\lambda \rho)$.
Like the proof of Lemma \ref{lem:ini_data} it is not hard to see for $t \in [0,t_*]$
\begin{align*}
\lVert D^k \bar{B}^\theta \rVert_{L_{r,z}^\infty}=\lVert D^k \bar{B}^\theta \rVert_{L^\infty} \lesssim \epsilon \lambda^{1-\sigma}\nu^{1/2} (\lambda\epsilon \lambda^{3-\sigma} \nu^{1/2}t_{*})^k =& \epsilon \lambda^{1-\sigma}\nu^{1/2} (\epsilon^{-1}\lambda)^k \\
=& \epsilon^{1-k}\lambda^{k-\sigma} \lambda \nu^{1/2}
\end{align*}
and with $|\mathrm{supp}_{r,z}| \sim \lambda^{-2}$, $|\mathrm{supp}| \sim \lambda^{-2}\nu^{-1}$,
\begin{equation} \label{eq:atilnorm-upbd}
\begin{split}
&\lVert D^k \bar{B}^\theta \rVert_{L_{r,z}^p} \leq |\mathrm{supp}_{r,z}|^{1/p} \lVert D^k \bar{B}^\theta \rVert_{L_{r,z}^\infty} \lesssim \epsilon^{1-k}\lambda^{k-\sigma} \lambda^{1-2/p} \nu^{1/2}, \\
&\lVert D^k \bar{B}^\theta \rVert_{L^p} \leq |\mathrm{supp}|^{1/p} \lVert D^k \bar{B}^\theta \rVert_{L^\infty} \lesssim \epsilon^{1-k}\lambda^{k-\sigma} \lambda^{1-2/p} \nu^{1/2-1/p}.
\end{split}
\end{equation}
We remark that an interpolation imply the same estimate when $k$ is not an integer. For the second part of the Lemma, we use the idea from \cite[Lemma 3.2, 3.3]{dai2.5D} of proving an upper bound of the dual Sobolev norm. Note when $t = t_*$,
\begin{equation}\label{eq:zako3}
\begin{split}
\lVert \bar{B}^\theta(t_{*}) \rVert_{L_{r,z}^2}^2 &\simeq \epsilon^2 \lambda^{2-2\sigma} \nu \int_{0}^{2\pi} \int_{0}^\infty g^2(\lambda \rho) \cos^2 \left(  \varphi - \epsilon^{-1}  \frac{f'(\lambda \rho)}{\lambda \rho} \right) \rho d\rho d\varphi \\
&= \epsilon^2 \lambda^{-2\sigma} \nu \int_{0}^{2\pi} \int_{0}^\infty g^2(\tilde \rho) \cos^2 \left( \varphi - \frac{\epsilon^{-1} f'(\tilde \rho)}{\tilde{\rho}} \right) \tilde \rho d\tilde\rho d\varphi \\
&\simeq \epsilon^2 \lambda^{-2\sigma} \nu \int_{0}^{2\pi} \int_{0}^\infty g^2(\tilde \rho) \tilde \rho d\tilde\rho d\varphi \\
&\simeq\epsilon^2\lambda^{-2\sigma} \nu.
\end{split}
\end{equation}
where we may integrate over $\rho$ from 0 to $\infty$ thanks to compact support, and we set $\tilde \rho = \lambda \rho$. Recall $f'=1$ on the the support of $g$ thus when $t=t_*$ we may write
\[
\bar{B}^\theta (t_{*},\rho,\varphi) = \epsilon \lambda^{1-\sigma}\nu^{1/2} g(\lambda\rho) \cos \left( \varphi - \frac{\epsilon^{-1}}{\lambda\rho} \right) =:\mathrm{Re} \ G(\tilde{\rho},\varphi)e^{i\psi(\tilde{\rho},\varphi)}.
\]
where $G(\tilde{\rho},\varphi) = \epsilon \lambda^{1-\sigma}\nu^{1/2}g(\tilde\rho)$ and $\psi(\tilde\rho,\varphi) = \varphi-\frac{\epsilon^{-1}}{\tilde\rho}$. Note
\[
\nabla_{\tilde{\rho},\varphi} \psi = \frac{\epsilon^{-1}}{\tilde \rho^2}e_{\rho} + \frac{1}{\tilde\rho}e_{\varphi}
\]
and on the support $\tilde{\rho} \sim 1$,
\[
|\nabla_{\tilde{\rho},\varphi} \psi|^2 \sim 1+\epsilon^{-2}.
\]
For all integers $m \ge 0$, similarly,
\[
|D_{\tilde{\rho},\varphi}^m \psi|^2 \sim_{m} 1+\epsilon^{-2}.
\]
Take $\mathcal L := \frac{\nabla_{\tilde{\rho},\varphi} \psi}{i|\nabla_{\tilde{\rho},\varphi} \psi|^2}\cdot\nabla_{\tilde{\rho},\varphi}$, where $\mathcal Le^{i\psi} = e^{i\psi}$. It is not hard to see on support $\tilde\rho\sim1$,
\[
\left|D_{\tilde{\rho},\varphi}^m\frac{\nabla_{\tilde{\rho},\varphi} \psi }{i|\nabla_{\tilde{\rho},\varphi} \psi |^2}\right|^2 \lesssim_{m} \frac{1}{1+\epsilon^{-2}}
\]
Then, for any $\phi \in C_c^\infty$, integrate by parts $k$ times (we may first assume $k \in \mathbb{N}$)
\[
 \int_{\mathbb{R}^2} Ge^{i\psi} \phi d\tilde{x}  = \int_{\mathbb{R}^2} e^{i\psi} \mathcal{L}_{*}^k (G\phi) d\tilde{x}
\]
where $\tilde{x} = (\nu^{-1} +\tilde \rho \cos \varphi,\tilde{\rho}\sin \varphi)$ and $\mathcal L_{*}\phi = \nabla_{\tilde{\rho},\varphi} \cdot \left(\frac{\nabla_{\tilde{\rho},\varphi} \psi}{-i|\nabla_{\tilde{\rho},\varphi} \psi|^2}\phi\right)$. Note derivatives on $G$ is harmless since $G$ is just a smooth function supported near $\tilde\rho \sim 1$. Also, derivatives on $\frac{\nabla_{\tilde{\rho},\varphi} \psi}{-i|\nabla_{\tilde{\rho},\varphi} \psi|^2}$ does not contribute a growing factor. So,
\[
\left| \int_{\mathbb{R}^2} Ge^{i\psi} \phi dx\right| \lesssim \lVert G \rVert_{L_{r,z}^2}  \frac{1}{(1+\epsilon^{-2})^{m/2}} \lVert \phi \rVert_{H_{r,z}^m} \lesssim \frac{\epsilon\lambda^{1-\sigma}\nu^{1/2}}{(1+\epsilon^{-2})^{k/2}} \lVert \phi \rVert_{H_{r,z}^k}.
\]
By duality,
\[
\lVert \bar{B}^\theta(t_{*},\tilde{\rho},\varphi) \rVert_{H_{r,z}^{-k}} \le \lVert Ge^{i\psi}(t,\tilde{\rho},\varphi) \rVert_{H_{r,z}^{-k}} \lesssim \frac{\epsilon\lambda^{1-\sigma}\nu^{1/2}}{(1+\epsilon^{-2})^{k/2}}.
\]
Note this also holds for non-integer $k$ by interpolation. Then, from \eqref{eq:zako3}, $\lVert \bar{B}^\theta(t_{*},\tilde{\rho},\varphi) \rVert_{L_{r,z}^2} = \epsilon\lambda^{1-\sigma}\nu^{\frac{1}{2}}$ and
\[
\lVert \bar{B}^\theta(t_{*},\tilde{\rho},\varphi) \rVert_{H_{r,z}^{k}} \gtrsim \lVert \bar{B}^\theta(t_{*},\tilde{\rho},\varphi) \rVert_{H_{r,z}^{-k}}^{-1} \lVert \bar{B}^\theta(t_{*},\tilde{\rho},\varphi) \rVert_{L_{r,z}^2}^2 \gtrsim \epsilon(1+\epsilon^{-2})^{k/2} \lambda^{1-\sigma}\nu^{1/2}.
\]
Also, because the $L^2$ norm is relatively small,
\[
\lVert \bar{B}^\theta(t_{*},\tilde{\rho},\varphi) \rVert_{\dot{H}_{r,z}^{k}} \gtrsim \epsilon(1+\epsilon^{-2})^{k/2} \lambda^{1-\sigma}\nu^{1/2}.
\]
Return to $(\rho,\varphi)$ coordinate,
\[
\lVert \bar{B}^\theta(t_{*},\rho,\varphi) \rVert_{\dot{H}_{r,z}^{k}} \simeq \lambda^{k-1}\lVert \bar{B}^\theta(t_{*},\tilde{\rho},\varphi) \rVert_{\dot{H}_{r,z}^{k}} \gtrsim  \epsilon(1+\epsilon^{-2})^{k/2} \lambda^{k-\sigma}\nu^{1/2}
\]
and when $k\in \mathbb{N}$,
\[
\lVert \bar{B}^\theta(t_{*},\rho,\varphi) \rVert_{\dot{H}^{k}} \gtrsim \nu^{-1/2} \lVert \bar{B}^\theta(t_{*},\rho,\varphi) \rVert_{\dot{H}_{r,z}^{k}} \gtrsim \epsilon(1+\epsilon^{-2})^{k/2} \lambda^{k-\sigma}.
\]
When $\epsilon>0$ is small, $(1+\epsilon^{-2})^{k/2}\sim \epsilon^{-k}$. Together with \eqref{eq:atilnorm-upbd}, we have
\begin{equation*}
\lVert \bar{B}^\theta(t_{*}) \rVert_{\dot{H}_{r,z}^{k}}\simeq\epsilon^{1-k}\lambda^{k-\sigma} \nu^{1/2}, \quad \lVert \bar{B}^\theta(t_{*}) \rVert_{\dot{H}^{k}}\simeq\epsilon^{1-k}\lambda^{k-\sigma}.
\end{equation*}
Lastly, since the gaining factor for $\bar B^\theta$ is $\lambda \gg \nu$ and $\bar B$ is supported near $r \sim \nu^{-1}$ thus $|D^k e_{\theta}| \lesssim_{k} \nu^k$, then
\[
\lVert D^k \bar{B} \rVert_{L^2} \lesssim \sum_{i=0}^k \nu^{i}\lVert D^{k-i}\bar{B}^\theta \rVert_{L^2}  \lesssim  \sum_{i=0}^k  \left( \frac{\nu}{\lambda} \right)^i \lVert D^k \bar{B}^\theta \rVert_{L^2} \lesssim \lVert D^k \bar{B}^\theta \rVert_{L^2}
\]
for vector-valued $\bar B = \bar B^\theta e_\theta$. For, the reverse inequality, we need to calculate the gaining factor for $\bar B$ briefly. Note in $z$-direction, $\mathrm{supp} \ \bar B$ is bounded with size $\sim \lambda^{-1}$. Then Poincaré inequality with sharp constant implies
\[
\lVert \bar{B} \rVert_{L^2} \lesssim \lambda^{-1} \lVert \partial_{z}\bar{B} \rVert_{L^2} \le \lambda^{-1} \lVert D \bar{B} \rVert_{L^2}  
\]
thus the gaining factor for $\bar B$ is also $\lambda$.
\[
\lVert D^k \bar{B}^\theta \rVert_{L^2} =  \lVert D^k (\bar{B}\cdot e_\theta) \rVert_{L^2} \lesssim \sum_{i=0}^k \nu^{i}\lVert D^{k-i}\bar{B} \rVert_{L^2} \lesssim  \sum_{i=0}^\infty  \left( \frac{\nu}{\lambda} \right)^i \lVert D^k \bar{B} \rVert_{L^2} \lesssim \lVert D^k \bar{B} \rVert_{L^2} .
\]
Therefore, $\lVert D^k \bar{B} \rVert_{L^2} \simeq \lVert D^k \bar{B}^\theta \rVert_{L^2}$, which is the last part of the lemma when $k$ is an integer. For non-integer $k$, the upper bound can be obtained from usual interpolation and lower bound follows from Gagliardo-Nirenberg
\[
\lVert \bar{B} \rVert_{\dot{H}^k} \gtrsim \frac{\lVert \bar{B} \rVert_{\dot{H}^m}^{m+1-k}}{\lVert \bar{B} \rVert_{\dot{H}^{m+1}}^{m-k}} 
\]
for integer $m > k$.
\end{proof}

From now on, we fix $\epsilon>0$ as a sufficiently small factor, thus below we do not keep track of $\epsilon$ or its dependence. Next we write the approximation system \eqref{pde:approx} in the form of \eqref{eq:hallMHD} with error terms. Although $u_{0}^r$ and $u_{0}^z$ form a stationary solution to 2D Euler equation in $(r,z)$ plane, our approximate solution $(\bar{u},\bar{B})$ is not a real solution to Hall MHD. Instead, plugging \eqref{pde:approx} in \eqref{pde:axisHallMHD}, we see
\begin{equation}\label{eq:hall-approx-axis}
\begin{split}
&\partial_{t}\bar{u}^r + (\bar{u}^r \partial_r + \bar{u}^z \partial_z) \bar{u}^r + \partial_r \bar{p} + \frac{(\bar{B}^\theta)^2}{r} = \bar{E}_{u}^r, \\[1em]
&\partial_t \bar{u}^z + (\bar{u}^r \partial_r + \bar{u}^z \partial_z) \bar{u}^z  + \partial_z \bar{p} = \bar{E}_{u}^z, \\[1em]
&\partial_t \bar{B}^\theta + (\bar{u}^r \partial_r + \bar{u}^z \partial_z) \bar{B}^\theta - \frac{\bar{u}^r \bar{B}^\theta}{r} - \partial_z \left( \frac{(\bar{B}^\theta)^2}{r} \right) = \bar{E}_{B}^\theta.
\end{split}
\end{equation}
where $\bar p = p_0(r,z)$ and
\begin{equation}\label{eq:Ebardef}
\begin{split}
&\bar{E}_{u}^r = u_{c}^z \partial_{z} u_{0}^r + \frac{(\bar{B}^\theta)^2}{r}, \\
&\bar{E}_{u}^z = u_{0}^r \partial_{r} u_{c}^z + u_{c}^z \partial_{z}u_{0}^z + u_{0}^z \partial_{z}u_{c}^z + u_{c}^z\partial_{z}u_{c}^z,\\
&\bar{E}_{B}^\theta = u_{c}^z \partial_{z} \bar{B}^\theta - \frac{u_{0}^r \bar{B}^\theta}{r} - \partial_z \left( \frac{(\bar{B}^\theta)^2}{r} \right).
\end{split}
\end{equation}
Thus with $\bar{E}_{u} = \bar{E}_u^r e_r + \bar{E}_u^z e_z$ and $\bar{E}_B = \bar{E}_B^\theta e_\theta$, $(\bar{u},\bar{B})$ satisfies
\begin{equation}\label{pde:ubarBbar}
\begin{split}
&\bar{u}_t + ( \bar{u} \cdot \nabla) \bar{u} - (\nabla \times  {\bar{B}}) \times  {\bar{B}} + \nabla \bar{p} = \bar{E}_{u},\\
&\bar{B}_t - \nabla \times ( \bar{u} \times  \bar{B}) +  \nabla \times ((\nabla \times  {\bar{B}}) \times  {\bar{B}}) =  \bar{E}_{B}.
\end{split}
\end{equation}

\begin{lemma}\label{lem:est Ebar}
For $k \ge 0$,
\begin{align*}
&\lVert \bar{E}_{u} \rVert_{\dot H^{k}} \lesssim \lVert \bar{E}_{u}^r \rVert_{\dot{H}^{k}} + \lVert \bar{E}_{u}^z \rVert_{\dot{H}^{k}}   \lesssim \lambda^{k+1-\sigma}(\lambda^{-1}\nu) (\lambda^{3-\sigma}\nu^{1/2}),\\
&\lVert \bar{E}_{B} \rVert_{\dot H^{k}} \simeq \lVert \bar{E}_{B}^\theta \rVert_{\dot H^{k}} \lesssim \lambda^{k-\sigma}(\lambda^{-1}\nu) (\lambda^{3-\sigma}\nu^{1/2}),
\end{align*}
where $\lVert \bar{E}_{B} \rVert_{\dot H^{k}} = \lVert \bar{E}_{B}^\theta e_{\theta} \rVert_{\dot H^{k}} \simeq \lVert \bar{E}_{B}^\theta \rVert_{\dot H^{k}}$ since they are supported near $r \sim \nu^{-1}$ (see the last part of proof to Lemma \ref{lem:Bbar}).
\end{lemma}
\begin{proof}
We shall estimate terms in \eqref{eq:Ebardef} one by one and will frequently use Lemma \ref{lem:ini_data} and Lemma \ref{lem:Bbar}. The estimates of $u_{c}^z \partial_{z} u_{0}^r$, $u_{0}^r \partial_{r} u_{c}^z$, $ u_{c}^z \partial_{z}u_{0}^z$, $ u_{0}^z \partial_{z}u_{c}^z$, and $ u_{c}^z\partial_{z}u_{c}^z$ are similar, so we only present the estimate $u_{c}^z \partial_{z} u_{0}^r$ here.
\begin{align*}
\lVert u_{c}^z \partial_{z} u_{0}^r \rVert_{\dot{H}^{k}} \lesssim& \sum_{i=0}^{k} \lVert D^i u_{c}^z \rVert_{L^\infty} \lVert D^{k-i+1}u_{0}^r \rVert_{L^2}  \\
\lesssim& \sum_{i=0}^{k} \lambda^{i+1-\sigma}(\lambda^{-1}\nu)(\lambda \nu^{1/2}) \cdot \lambda^{k-i+2-\sigma}\\
\lesssim&  \lambda^{k+1-\sigma} (\lambda^{-1}\nu) (\lambda^{3-\sigma}\nu^{1/2}).
\end{align*}
Recall $r^{-1} \sim \nu \ll \lambda$, thus
\begin{align*}
\left\lVert  \frac{(\bar{B}^\theta)^2}{r}  \right\rVert_{\dot{H}^{k}} \lesssim \nu  \lVert (\bar{B}^\theta)^2 \rVert_{\dot{H}^{k}} &\lesssim \nu \sum_{i=0}^{k} \lVert D^i \bar{B}^\theta \rVert_{L^\infty} \lVert D^{k-i}\bar{B}^\theta \rVert_{L^2}\\
&\lesssim \nu \sum_{i=0}^{k} \lambda^{i-\sigma}(\lambda \nu^{1/2}) \lambda^{k-i-\sigma}\\
&\lesssim \lambda^{k+1-\sigma} (\lambda^{-1}\nu) (\lambda^{3-\sigma}\nu^{1/2}) \lambda^{-2}\\
&\lesssim \lambda^{k+1-\sigma} (\lambda^{-1}\nu) (\lambda^{3-\sigma}\nu^{1/2}).
\end{align*}
where in the last step we used \eqref{eq:nu-lam}. Note we cheat a little since it is possible that derivative acts on $\frac{1}{r}$ and can produce $\frac{1}{r^m}$. However, in that case, each derivative on $\frac{1}{r}$ produces a factor of $\nu \ll \lambda$ thus could be absorbed by estimate above. Therefore, $\lVert \bar{E}_{u} \rVert_{\dot H^{k-1}} \lesssim \lambda^{k-\sigma}(\lambda^{-1}\nu) (\lambda^{3-\sigma}\nu^{1/2})$.

Next,
\begin{align*}
\lVert u_{c}^z \partial_{z} \bar{B}^\theta \rVert_{\dot H^k} \lesssim& \sum_{i=0}^k \lVert D^i u_{c}^z \rVert_{L^\infty} \lVert D^{k-i+1} \bar{B}^\theta \rVert_{L^2}  \\
\lesssim& \sum_{i=0}^k \lambda^{i+1-\sigma}(\lambda^{-1}\nu) \lambda \nu^{1/2} \cdot \lambda^{k-i+1-\sigma}\\
\lesssim& \lambda^{k+3-2\sigma} \nu^{1/2} (\lambda^{-1}\nu) = \lambda^{k-\sigma}  (\lambda^{-1}\nu) (\lambda^{3-\sigma}\nu^{1/2}).
\end{align*}
and
\begin{align*}
\left\lVert \partial_{z} \frac{(\bar{B}^\theta)^2}{r}  \right\rVert_{\dot{H}^{k}} \lesssim \lambda \left\lVert  \frac{(\bar{B}^\theta)^2}{r}  \right\rVert_{\dot{H}^{k}} \lesssim& \lambda^{k-\sigma} (\lambda^{-1}\nu) (\lambda^{3-\sigma}\nu^{1/2}).
\end{align*}
Also by the same reason $r^{-1} \sim \nu$,
\begin{align*}
\left\lVert  \frac{u_{0}^r \bar{B}^\theta}{r}  \right\rVert_{\dot{H}^k} \le \nu \lVert u_{0}^r \bar{B}^\theta \rVert_{\dot{H}^k} \le& \nu\sum_{i=0}^k \lVert D^i u_{0}^r \rVert_{L^\infty} \lVert D^{k-i} \bar{B}^\theta \rVert_{L^2} \\
\lesssim& \nu\sum_{i=0}^k  \lambda^{i+1-\sigma} (\lambda \nu^{1/2}) \cdot \lambda^{k-i-\sigma}\\
\lesssim& \lambda^{k-\sigma} (\lambda^{-1}\nu) (\lambda^{3-\sigma}\nu^{1/2}).
\end{align*}
So, $\lVert \bar{E}_{B} \rVert_{\dot H^{k}} \lesssim \lambda^{k-\sigma}(\lambda^{-1}\nu) (\lambda^{3-\sigma}\nu^{1/2})$.
\end{proof}

\section{Control of Perturbation under no Dissipation}\label{sec-axi-error}
We first assume the existence of smooth solution $(u,B)$ on $[0,t_{*}]$ and leave the justification in Section \ref{subsec-pf-axis}. Before controlling the difference $u^r - \bar u^r$, $u^z - \bar u^z$ , and $B^\theta - \bar B^\theta$, we make the following observation which would be useful in later estimates.
\begin{lemma}\label{lem-bdd-supp}
For the solution $(u,B)$ of \eqref{pde:axisHallMHD}, if
\[
\lVert u^r \rVert_{L^\infty} \lesssim \lambda^{2-\sigma}\nu^{1/2},\quad \lVert u^z \rVert_{L^\infty} \lesssim \lambda^{2-\sigma}\nu^{1/2}, \quad \left\lVert  \frac{B^\theta}{r}  \right\rVert_{L^\infty}    \lesssim \lambda^{2-\sigma}\nu^{1/2},
\]
then $\forall t \in [0,t_0]$ for $t_0 \le t_*$, and $x_{0} \in \mathrm{supp} \ B^\theta(t)$, 
\[
\mathrm{dist} (x_{0}, \{(r,z): r=\nu^{-1}, z=0 \} ) \lesssim \lambda^{-1}.
\]
In other words, 
\[
\mathrm{dist}(\mathrm{supp} \ B^\theta(t), \{(r,z): r=0\}) \simeq \nu^{-1}.
\]
In particular, since $\bar B^\theta (t)$ is supported in a disk of radius $\lambda^{-1}$ centered at $(r,z)= (\nu^{-1},0)$, $\forall t \in [0,t_0]$, if $\mathrm{supp} \ H^\theta(t) \neq \emptyset$, then
\[
\mathrm{dist}(\mathrm{supp} \ H^\theta(t), \{(r,z): r=0\}) \simeq \nu^{-1}.
\]
\end{lemma}
\begin{proof}
We first rewrite the $B$-equation in \eqref{pde:axisHallMHD}
\begin{equation}\label{eq:B-trans}
\partial_{t}B^\theta + V \cdot \nabla B^\theta = \frac{u^r B^\theta}{r}, \qquad B^\theta(0) = B^\theta_{0}.
\end{equation}
where $V= u^re_{r} + (u^z - \frac{2B^\theta}{r})e_z$. Note \eqref{eq:B-trans} is a transport equation with characteristics along $V$, and on the characteristics $B^\theta$ evolves under source term $ \frac{u^r B^\theta}{r}$.
Thus from the assumption,
\[
\lVert V \rVert_{L_{r,z}^\infty} \lesssim  \lambda^{2-\sigma}\nu^{1/2}
\]
and before $t_0\le t_{*} \simeq \lambda^{\sigma-3} \nu^{-1/2}$, the characteristics at most propagate a distance $\sim \lambda^{-1}$. Since $B_{0}^\theta$ \eqref{ini_data} is supported in a disk of radius around $\lambda^{-1}$ centered at $(r,z)=(\nu^{-1},0)$, the lemma is proved.
\end{proof}

Therefore, it is convenient to introduce a new scalar $a := B^\theta/r$ and note $\mathrm{supp} \ a = \mathrm{supp} \ B^\theta$. Then \eqref{pde:axisHallMHD} turns to a more compact form
\begin{equation}\label{pde:ua}
\begin{split}
&\partial_{t} u + u \cdot \nabla u + \nabla p = -ra^2e_{r}, \quad u(0) = u_{0}^r e_{r} + (u_c^z+u_{0}^z) e_{z}, \\
&\partial_{t}a + (u-2ae_{z})\cdot \nabla a = 0, \quad a(0)= a_{0} := \frac{B_{0}^\theta}{r}, \\
&\nabla \cdot u = 0.
\end{split}
\end{equation}
We also set $\bar a = \frac{\bar B^\theta}{r}$, whose approximation system is
\begin{equation}\label{pde:ua-bar}
\begin{split}
&\partial_{t} \bar{u} + \bar{u} \cdot \nabla \bar{u} + \nabla \bar{p} = -r\bar{a}^2e_{r} + \bar{E}_{u}, \quad \bar{u}(0) = u(0), \\
&\partial_{t}\bar{a} + (\bar{u}-2\bar{a}e_{z})\cdot \nabla \bar{a} = \frac{\bar{E}_{B}^\theta}{r}, \quad \bar{a}(0)= a(0), \\
&\nabla \cdot \bar{u} = 0.
\end{split}
\end{equation}
Although we do not have the divergence-free property of $a$, the $a$-equation is now a pure transport equation thus avoiding loss of derivatives in general 3D Hall MHD \eqref{eq:hallMHD}. Set $w = u - \bar u$, $h = a - \bar a$. Since $w(0)=h(0) = 0$, for a small time $t_{0}>0$ and $0 \le k \le 10$
\begin{equation}\label{boot}
\lVert w(t) \rVert_{\dot{H}^k} \le 2 \lambda^{k+1-\sigma}(\lambda^{-1}\nu)^{1-\phi}, \quad \lVert h(t) \rVert_{\dot{H}^k} \le 2 \lambda^{k+1-\sigma}(\lambda^{-1}\nu)^{2-\phi}, \quad \forall t \in [0,t_{0}].
\end{equation}
where $\phi>0$ is a very small constant and it suffices to take $\phi = 0.1$. We control $w$ and $h$ via a bootstrap argument. That is, we shall prove
\begin{equation}\label{boot-goal}
\lVert w(t) \rVert_{\dot{H}^k} \le  \lambda^{k+1-\sigma}(\lambda^{-1}\nu)^{1-\phi}, \quad \lVert h(t) \rVert_{\dot{H}^k} \le  \lambda^{k+1-\sigma}(\lambda^{-1}\nu)^{2-\phi}, \quad \forall t \in [0,t_{0}],
\end{equation}
then a continuity argument extends \eqref{boot} to $[0,t_*]$.

Take difference between \eqref{pde:ua} and \eqref{pde:ua-bar},
\begin{equation}\label{pde:wh}
\begin{split}
&\partial_t w +\bar u\cdot\nabla w +w\cdot\nabla\bar u + w\cdot\nabla w +\nabla q +rh(h+2\bar a)e_r = -\bar E_u, \\
&\partial_{t}h + (\bar{u}-2\bar{a}e_{z})\cdot \nabla h  + (w-2he_{z})\cdot \nabla \bar{a} + (w-2he_{z})\cdot \nabla h = -\frac{\bar{E}_{B}^\theta}{r}, \\
&\nabla \cdot w = 0.
\end{split}
\end{equation}
where $q = p -\bar p$. Thanks to bounded support of $\bar B^\theta$ such that $r \sim \nu^{-1}$, Lemma \ref{lem:Bbar} implies
\begin{equation}\label{eq:abar-est}
\lVert \bar{a} \rVert_{W^{k,p}} \lesssim \lambda^{k-\sigma}\lambda^{1-2/p}\nu^{3/2-1/p}. 
\end{equation}
It is also not hard to see under \eqref{boot}, the assumption of Lemma \ref{lem-bdd-supp} is satisfied on $[0,t_{0}]$ so $a$ has bounded support near $r \sim \nu^{-1}$: By Sobolev embedding and small $\delta > 0$
\begin{align*}
&\lVert w \rVert_{L^\infty} \lesssim \lVert w \rVert_{H^{3/2+\delta}} \lesssim \lambda^{5/2+\delta-\sigma}(\lambda^{-1}\nu)^{1-\phi} \lesssim \lambda^{2-\sigma}\nu^{1/2},\\
&\lVert h \rVert_{L^\infty} \lesssim \lVert h \rVert_{H^{3/2+\delta}} \lesssim \lambda^{5/2+\delta-\sigma}(\lambda^{-1}\nu)^{2-\phi} \lesssim \lambda^{1-\sigma}\nu^{3/2},
\end{align*}
where we should note that the inhomogeneous norms is used while we only know the homogeneous norms \eqref{boot}. But it is not hard to see from \eqref{boot} that, adding homogeneous norms together, the inhomogeneous norm is dominated by the homogeneous norm, so 
\[
\lVert w(t) \rVert_{H^k} \lesssim \lambda^{k+1-\sigma}(\lambda^{-1}\nu)^{1-\phi}, \quad \lVert h(t) \rVert_{H^k} \lesssim \lambda^{k+1-\sigma}(\lambda^{-1}\nu)^{2-\phi}, \quad \forall t \in [0,t_{0}].
\]
and we shall omit similar argument below. By estimate of $\bar u = u_0$ from Lemma \ref{lem:ini_data} and estimate of $\bar a$ \eqref{eq:abar-est},
\begin{align*}
&\lVert u \rVert_{L^\infty} \le \lVert w \rVert_{L^\infty} + \lVert \bar{u} \rVert_{L^\infty}  \lesssim \lambda^{2-\sigma}\nu^{1/2},\\
&\lVert a \rVert_{L^\infty} \le \lVert h \rVert_{L^\infty} + \lVert \bar{a} \rVert_{L^\infty}  \lesssim \lambda^{1-\sigma}\nu^{3/2} < \lambda^{2-\sigma}\nu^{1/2}.
\end{align*}
Lastly, \eqref{boot}, Lemma \ref{lem:ini_data} and \eqref{eq:abar-est} imply that on $[0,t_{0}]$
\begin{equation}\label{eq:ua-est}
\lVert u (t) \rVert_{\dot{H}^k} \lesssim \lambda^{k+1-\sigma}, \quad \lVert a(t) \rVert_{\dot{H}^k} \lesssim \lambda^{k-\sigma} \nu.
\end{equation}

Now we start to prove \eqref{boot-goal}, and actually we could prove something stronger.

\begin{lemma}\label{lem-control-error}
Under the bootstrap assumption \eqref{boot}, for $0 \le k \le 10$,
\[
\lVert w \rVert_{\dot{H}^k} \lesssim \lambda^{k+1-\sigma} (\lambda^{-1}\nu)^{1-\phi+\phi^{k+1}}, \quad \lVert h \rVert_{\dot{H}^k} \lesssim \lambda^{k+1-\sigma} (\lambda^{-1}\nu)^{2-\phi+\phi^{k+1}}, \quad \forall t \in [0,t_{0}].
\]
\end{lemma}
\begin{proof}
From \eqref{boot} and Sobolev embedding, for $0\le m \le 8$ and small $0<\delta<\left( \frac{1}{2}-\phi \right)\frac{7/2-\sigma}{100}$
\begin{equation}\label{eq:whDooo}
\begin{split}
&\lVert D^m w \rVert_{L^\infty} \lesssim \lVert w \rVert_{H^{m+3/2+\delta}} \lesssim \lambda^{m+5/2+\delta-\sigma}(\lambda^{-1}\nu)^{1-\phi} \lesssim \lambda^{m+2-\sigma}\nu^{1/2},\\
&\lVert D^m h \rVert_{L^\infty} \lesssim \lVert h \rVert_{H^{m+3/2+\delta}} \lesssim \lambda^{m+5/2+\delta-\sigma}(\lambda^{-1}\nu)^{2-\phi} \lesssim \lambda^{m+1-\sigma}\nu^{3/2}.
\end{split}
\end{equation}
We will prove by induction, starting from $L^2$ case and inductively proving $\dot H^k$ case.

\noindent $\mathbf{L^2}$ \textbf{case:} Multiply the $w$-equation in \eqref{pde:wh} by $w$, integrate, and apply H\"older inequality
\begin{align*}
\frac{1}{2}\partial_{t}\lVert w \rVert_{L^2}^2 \leq 2\lVert \nabla \bar{u} \rVert_{L^\infty}\lVert w \rVert_{L^2}^2 + \lVert \nabla w \rVert_{L^\infty}\lVert w \rVert_{L^2}^2 + (\lVert r(h+2\bar{a}) \rVert_{L^\infty} \lVert h \rVert_{L^2} + \lVert \bar{E}_{u} \rVert_{L^2}  )\lVert w \rVert_{L^2}, 
\end{align*}
where we do not rely on divergence-free property of $\bar u$ and cancellation of Euler nonlinearity. Although it is doable, we point out it is not necessary. Instead, integration by parts is enough here, which makes later estimate of $h$ possible since we do not have divergence-free property of $a$ or $h$. Divide above by $\|w\|_{L^2}$, 
\[
\partial_{t}\lVert w \rVert_{L^2} \leq (2\lVert \nabla \bar{u} \rVert_{L^\infty}+ \lVert \nabla w \rVert_{L^\infty})\lVert w \rVert_{L^2} + \lVert r(h+2\bar{a}) \rVert_{L^\infty} \lVert h \rVert_{L^2} + \lVert \bar{E}_{u} \rVert_{L^2}.
\]
From Lemma \ref{lem:Bbar} and \eqref{eq:whDooo},
\begin{equation}\label{eq:zako1}
2\lVert \nabla \bar{u} \rVert_{L^\infty}+ \lVert \nabla w \rVert_{L^\infty} \lesssim \lambda^{3-\sigma}\nu^{1/2}.
\end{equation}
Also, apply bounded support of $h$ and $\bar a$ from Lemma \ref{lem-bdd-supp}
\begin{equation}\label{eq:zako2}
\lVert r(h+2\bar{a}) \rVert_{L^\infty} \lesssim \nu^{-1} \lambda^{1-\sigma}\nu^{3/2} \le \lambda^{1-\sigma}\nu^{1/2}.
\end{equation}
Together with Lemma \ref{lem:est Ebar},
\begin{equation}\label{eq:est-w-2}
\partial_{t}\lVert w \rVert_{L^2} \lesssim \lambda^{3-\sigma}\nu^{1/2}\lVert w \rVert_{L^2} + \lambda^{1-\sigma}\nu^{1/2} \lVert h \rVert_{L^2} + \lambda^{1-\sigma}(\lambda^{-1}\nu)(\lambda^{3-\sigma}\nu^{1/2}).
\end{equation}
Similarly, for $h$-equation,
\begin{equation}\label{eq:est-h-2}
\begin{split}
\partial_{t}\lVert h \rVert_{L^2} \le& \lVert \nabla (\bar{u} - 2\bar{a}e_{z}) \rVert_{L^\infty}\lVert h \rVert_{L^2} + \lVert \nabla \bar{a} \rVert_{L^\infty}\lVert w \rVert_{L^2} + 2 \lVert \nabla \bar{a} \rVert_{L^\infty} \lVert h \rVert_{L^2} \\
&+ \lVert \nabla (w-2he_{z}) \rVert_{L^\infty}\lVert h \rVert_{L^2} + \left\lVert  \frac{\bar{E}_{B}^\theta}{r}  \right\rVert_{L^2}     \\
\lesssim& \lambda^{3-\sigma}\nu^{1/2}\lVert h \rVert_{L^2} + \lambda^{2-\sigma}\nu^{3/2}\lVert w \rVert_{L^2} + \lambda^{1-\sigma}(\lambda^{-1}\nu)^2(\lambda^{3-\sigma}\nu^{1/2})  
\end{split}
\end{equation}
where from bounded support of approximate solutions and Lemma \ref{lem:est Ebar}, we have $\left\lVert  \frac{\bar{E}_{B}^\theta}{r}  \right\rVert_{L^2} \lesssim \nu \|\bar E_B^\theta\|_{L^2} \lesssim \lambda^{1-\sigma}(\lambda^{-1}\nu)^2(\lambda^{3-\sigma}\nu^{1/2})  $. Set $W(t) = \lambda^{-1}\nu\|w(t)\|_{L^2}$. Multiply \eqref{eq:est-w-2} by $\lambda^{-1}\nu$ and add \eqref{eq:est-h-2},
\[
\partial_{t} (W + \lVert h \rVert_{L^2} ) \lesssim \lambda^{3-\sigma}\nu^{1/2} W + \lambda^{3-\sigma}\nu^{1/2}\lVert h \rVert_{L^2} + \lambda^{1-\sigma}(\lambda^{-1}\nu)^2(\lambda^{3-\sigma}\nu^{1/2}).
\]
By Gr\"onwall inequality, on $[0,t_{0}]$ with $t_{0}\le t_* \lesssim \lambda^{\sigma-3}\nu^{-1/2}$
\begin{align*}
W(t) + \lVert h(t) \rVert_{L^2} &\lesssim e^{Ct\lambda^{3-\sigma}\nu^{1/2}}t\lambda^{1-\sigma}(\lambda^{-1}\nu)^2(\lambda^{3-\sigma}\nu^{1/2}) \\
&\lesssim \lambda^{1-\sigma}(\lambda^{-1}\nu)^2.
\end{align*}
In other words,
\[
\lVert w(t) \rVert_{L^2} \lesssim \lambda^{1-\sigma}(\lambda^{-1}\nu), \quad \lVert h(t) \rVert_{L^2} \lesssim  \lambda^{1-\sigma}(\lambda^{-1}\nu)^2.
\]

\noindent $\mathbf{\dot H^k}$ \textbf{case:} Assume Lemma \ref{lem-control-error} holds for $H^m$ with $m \le k-1$. Differentiate the $w$-equation in \eqref{pde:wh} to $k$th order, multiply by $D^kw$, integrate, and apply H\"older inequality,
\begin{align*}
\frac{1}{2}\partial_{t} \lVert D^k w \rVert_{L^2}^2 \le& \sum_{i = 1}^k \lVert D^i \bar{u} \rVert_{L^\infty} \lVert D^{k-i+1}w \rVert_{L^2}\lVert D^k w \rVert_{L^2} + \sum_{i = 0}^k \lVert D^{i+1} \bar{u} \rVert_{L^\infty} \lVert D^{k-i}w \rVert_{L^2}\lVert D^k w \rVert_{L^2} \\
&+ \sum_{i = 1}^{\min \{k,8\}} \lVert D^{i} w \rVert_{L^\infty} \lVert D^{k-i+1}w \rVert_{L^2}\lVert D^k w \rVert_{L^2} \\
&+ \sum_{i=\min \{k+1,9\}}^k    \lVert D^i w \rVert_{L^2}  \lVert D^{k-i+1}w \rVert_{L^\infty}\lVert D^k w \rVert_{L^2} \\
&+ \sum_{i=0}^{\min \{k,8\}} \lVert D^i r(h+2\bar{a}) \rVert_{L^\infty} \lVert D^{k-i}h \rVert_{L^2}\lVert D^k w \rVert_{L^2}   \\
&+ \sum_{i=\min \{k+1,9\}}^{k} \lVert D^i r(h+2\bar{a}) \rVert_{L^2} \lVert D^{k-i}h \rVert_{L^\infty}\lVert D^k w \rVert_{L^2}   \\
&+ \lVert \bar{E}_{u} \rVert_{\dot{H}^k} \lVert D^kw \rVert_{L^2},  
\end{align*}
where we decompose the estimate of $\|D^k(w\cdot\nabla w)\|_{L^2}$ and $\|D^kr(h+2\bar{a})h\|_{L^2}$ so that \eqref{eq:whDooo} could still be applied. Divide above by $\|D^k w\|_{L^2}$ and rearrange the summation,
\begin{align*}
&\partial_{t} \lVert D^k w \rVert_{L^2} \\
\le&  (2\lVert D \bar{u} \rVert_{L^\infty} + \lVert D w \rVert_{L^\infty} ) \lVert D^k w \rVert_{L^2} + \lVert r(h+2\bar{a}) \rVert_{L^\infty} \lVert D^k h \rVert_{L^2} + \lVert D^k r(h+2\bar{a}) \rVert_{L^2}\lVert h \rVert_{L^\infty}    \\
&+   \sum_{i = 2}^k \lVert D^i \bar{u} \rVert_{L^\infty} \lVert D^{k-i+1}w \rVert_{L^2} + \sum_{i = 1}^k \lVert D^{i+1} \bar{u} \rVert_{L^\infty} \lVert D^{k-i}w \rVert_{L^2} \\
&+ \sum_{i = 2}^{\min \{k,8\}} \lVert D^{i} w \rVert_{L^\infty} \lVert D^{k-i+1}w \rVert_{L^2} + \sum_{i=\min \{k+1,9\}}^{k-1}    \lVert D^i w \rVert_{L^2}  \lVert D^{k-i+1}w \rVert_{L^\infty}\\
&+ \sum_{i=1}^{\min \{k,8\}} \lVert D^i r(h+2\bar{a}) \rVert_{L^\infty} \lVert D^{k-i}h \rVert_{L^2} + \sum_{i=\min \{k+1,9\}}^{k-1} \lVert D^i r(h+2\bar{a}) \rVert_{L^2} \lVert D^{k-i}h \rVert_{L^\infty}   \\
&+ \lVert \bar{E}_{u} \rVert_{\dot{H}^k}.
\end{align*}
For the top order terms, note from \eqref{eq:whDooo} and \eqref{eq:abar-est}
\begin{align*}
\lVert D^k r(h+2\bar{a}) \rVert_{L^2}\lVert h \rVert_{L^\infty} \lesssim& \nu^{-1} \lVert h \rVert_{L^\infty}  \lVert D^k h \rVert_{L^2} + \nu^{-1}\lVert D^k \bar{a} \rVert_{L^2}\lVert h \rVert_{L^\infty}  \\
\lesssim& \lambda^{1-\sigma}\nu^{1/2}\lVert D^k h \rVert_{L^2} + \lambda^{k-\sigma} \lambda^{1-\sigma}\nu^{3/2}\\
\lesssim& \lambda^{1-\sigma}\nu^{1/2}\lVert D^k h \rVert_{L^2} + \lambda^{k+1-\sigma}(\lambda^{-1}\nu)(\lambda^{3-\sigma}\nu^{1/2})\lambda^{-2}
\end{align*}
where we cheat a little since derivative could be applied to $r$. However, Derivative on $r$ produces a gaining factor of $\nu$ while on $h$ and $\bar a$ its produces a gaining factor $\lambda \gg \nu$ thus the former case is absorbed.
Together with \eqref{eq:zako1} and \eqref{eq:zako2},
\begin{align*}
&(2\lVert D \bar{u} \rVert_{L^\infty} + \lVert D w \rVert_{L^\infty} ) \lVert D^k w \rVert_{L^2} + \lVert r(h+2\bar{a}) \rVert_{L^\infty} \lVert D^k h \rVert_{L^2} + \lVert D^k r(h+2\bar{a}) \rVert_{L^2}\lVert h \rVert_{L^\infty}\\
\lesssim& \lambda^{3-\sigma}\nu^{1/2}\lVert D^kw \rVert_{L^2} + \lambda^{1-\sigma}\nu^{1/2}\lVert D^k h \rVert_{L^2} + \lambda^{k+1-\sigma}(\lambda^{-1}\nu)(\lambda^{3-\sigma}\nu^{1/2})\lambda^{-2}.
\end{align*}
For the other lower order terms, we first notice from Lemma \ref{lem:est Ebar}
\[
\lVert \bar{E}_{u} \rVert_{\dot{H}^k} \lesssim  \lambda^{k+1-\sigma}(\lambda^{-1}\nu)(\lambda^{3-\sigma}\nu^{1/2}).
\]
Also, from $\bar u = u_0$, Lemma \ref{lem:ini_data}, and induction for $m \le k-1$
\begin{align*}
\sum_{i = 2}^k \lVert D^i \bar{u} \rVert_{L^\infty} \lVert D^{k-i+1}w \rVert_{L^2} &\lesssim \sum_{i=2}^k \lambda^{i+2-\sigma}\nu^{1/2}\lambda^{k-i+2-\sigma}(\lambda^{-1}\nu)^{1-\phi+\phi^{k-i+2}} \\
&\lesssim \lambda^{k+1-\sigma}(\lambda^{-1}\nu)^{1-\phi+\phi^{k}}(\lambda^{3-\sigma}\nu^{1/2}).
\end{align*}
and
\begin{align*}
\sum_{i=1}^{\min \{k,8\}} \lVert D^i r(h+2\bar{a}) \rVert_{L^\infty} \lVert D^{k-i}h \rVert_{L^2} &\lesssim \sum_{i=1}^{\min \{k,8\}} \lambda^{i+1-\sigma}\nu^{1/2} \lambda^{k-i+1-\sigma}(\lambda^{-1}\nu)^{2-\phi+\phi^{k-i+1}} \\
&\lesssim \lambda^{k+1-\sigma}(\lambda^{-1}\nu)^{2-\phi+\phi^{k}}(\lambda^{3-\sigma}\nu^{1/2})\lambda^{-2}.
\end{align*}
Others could be estimated similarly and dominated by $\lambda^{k+1-\sigma}(\lambda^{-1}\nu)^{1-\phi+\phi^{k}}(\lambda^{3-\sigma}\nu^{1/2})$. So, combine our estimates,
\begin{equation}\label{eq:w-k}
\partial_{t}\lVert D^k w \rVert_{L^2} \lesssim  \lambda^{3-\sigma}\nu^{1/2}\lVert D^kw \rVert_{L^2} + \lambda^{1-\sigma}\nu^{1/2}\lVert D^k h \rVert_{L^2} + \lambda^{k+1-\sigma}(\lambda^{-1}\nu)^{1-\phi+\phi^{k}}(\lambda^{3-\sigma}\nu^{1/2}).
\end{equation}
For $h$-equation, we use integration and H\"older inequality similarly,
\begin{align*}
\partial_{t}\lVert D^k h \rVert_{L^2} \le& \sum_{i=1}^k \lVert D^i (\bar{u}-2\bar{a}e_{z}) \rVert_{L^\infty} \lVert D^{k+1-i}h \rVert_{L^2} + \sum_{i=0}^k \lVert D^{i+1}\bar{a} \rVert_{L^\infty}\lVert D^{k-i}(w-2he_{z}) \rVert_{L^2}     \\
&+ \sum_{i=1}^{\min\{k,8\}} \lVert D^i (w-2he_{z}) \rVert_{L^\infty} \lVert D^{k+1-i}h \rVert_{L^2}  \\
&+ \sum_{i=\min\{k+1,9\}}^k \lVert D^i (w-2he_{z}) \rVert_{L^2} \lVert D^{k+1-i}h \rVert_{L^\infty} + \left\lVert  \frac{\bar{E}_{B}^\theta}{r}  \right\rVert_{\dot{H}^k}.
\end{align*}
Pulling out the top-order terms in the first line, we have
\begin{align*}
\partial_{t}\lVert D^k h \rVert_{L^2}\le& (\lVert D\bar{u} \rVert_{L^\infty}+ 4\lVert D\bar{a} \rVert_{L^\infty} + \lVert Dw \rVert_{L^\infty} + 4 \lVert D h \rVert_{L^\infty}  )\lVert D^kh \rVert_{L^2} \\
&+ (\lVert D \bar{a} \rVert_{L^\infty}+\lVert D h \rVert_{L^\infty})\lVert D^k w \rVert_{L^2}\\
&+\sum_{i=2}^k \lVert D^i (\bar{u}-2\bar{a}e_{z}) \rVert_{L^\infty} \lVert D^{k+1-i}h \rVert_{L^2} + \sum_{i=1}^k \lVert D^{i+1}\bar{a} \rVert_{L^\infty}\lVert D^{k-i}(w-2he_{z}) \rVert_{L^2}     \\
&+ \sum_{i=2}^{\min\{k,8\}} \lVert D^i (w-2he_{z}) \rVert_{L^\infty} \lVert D^{k+1-i}h \rVert_{L^2}  \\
&+ \sum_{i=\min\{k+1,9\}}^{k-1} \lVert D^i (w-2he_{z}) \rVert_{L^2} \lVert D^{k+1-i}h \rVert_{L^\infty} + \left\lVert  \frac{\bar{E}_{B}^\theta}{r}  \right\rVert_{\dot{H}^k}
\end{align*}
From $\bar u = u_0$, Lemma \ref{lem:ini_data}, \eqref{eq:abar-est}, and \eqref{eq:whDooo}, the top order terms
\begin{align*}
&(\lVert D\bar{u} \rVert_{L^\infty}+ 4\lVert D\bar{a} \rVert_{L^\infty} + \lVert Dw \rVert_{L^\infty} + 4 \lVert D h \rVert_{L^\infty}  )\lVert D^kh \rVert_{L^2} + (\lVert D \bar{a} \rVert_{L^\infty}+\lVert D h \rVert_{L^\infty})\lVert D^k w \rVert_{L^2}\\
\lesssim& \lambda^{2-\sigma}\nu^{3/2}\lVert D^kw \rVert_{L^2} + \lambda^{3-\sigma}\nu^{1/2}\lVert D^k h \rVert_{L^2}.
\end{align*}
Like in $w$ case, we estimate a few low order terms as examples instead of estimating all of them. From Lemma \ref{lem:est Ebar} and bounded support,
\[
\lVert D^k \frac{\bar{E}_{B}^\theta}{r} \rVert_{L^2} \lesssim \nu \lVert D^k \bar{E}_{B}^\theta \rVert_{L^2} \lesssim \lambda^{k+1-\sigma} (\lambda^{-1}\nu)^2 (\lambda^{3-\sigma}\nu^{1/2}),  
\]
where the case when the derivative is on $\frac{1}{r}$ is skipped as well, since it produces a smaller gaining factor $\nu$ compared to $\lambda$. From $\bar u = u_0$, Lemma \ref{lem:ini_data}, \eqref{eq:abar-est}, and induction for $m \le k-1$,
\begin{align*}
&\sum_{i=2}^k \lVert D^i (\bar{u}-2\bar{a}e_{z}) \rVert_{L^\infty} \lVert D^{k+1-i}h \rVert_{L^2}  \\
\lesssim& \sum_{i=2}^k (\lambda^{i+2-\sigma}\nu^{1/2}+ 2\lambda^{i+1-\sigma}\nu^{3/2})\lambda^{k+2-i-\sigma}(\lambda^{-1}\nu)^{2-\phi+\phi^{k+2-i}}\\
\lesssim& \lambda^{k+1-\sigma} (\lambda^{-1}\nu)^{2-\phi+\phi^k} (\lambda^{3-\sigma}\nu^{1/2}).
\end{align*}
Other low order terms could be estimated similarly and are dominated by $\lambda^{k+1-\sigma} (\lambda^{-1}\nu)^{2-\phi+\phi^k} (\lambda^{3-\sigma}\nu^{1/2})$. So,
\begin{equation}\label{eq:est-h-k}
\partial_{t}\lVert D^k h \rVert_{L^2} \lesssim \lambda^{2-\sigma}\nu^{3/2}\lVert D^kw \rVert_{L^2} + \lambda^{3-\sigma}\nu^{1/2}\lVert D^k h \rVert_{L^2} + \lambda^{k+1-\sigma} (\lambda^{-1}\nu)^{2-\phi+\phi^k} (\lambda^{3-\sigma}\nu^{1/2}).
\end{equation}
Set $W_{k} = \lambda^{-1}\nu \|D^k w\|_{L^2}$, add $\lambda^{-1}\nu$ times \eqref{eq:w-k} to \eqref{eq:est-h-k},
\[
\partial_{t}(W_{k} + \lVert D^kh \rVert_{L^2} ) \lesssim \lambda^{3-\sigma}\nu^{1/2}W_{k} + \lambda^{3-\sigma}\nu^{1/2}\lVert D^k h \rVert_{L^2} +  \lambda^{k+1-\sigma} (\lambda^{-1}\nu)^{2-\phi+\phi^k} (\lambda^{3-\sigma}\nu^{1/2}),
\]
and by Gr\"onwall inequality on $[0,t_{0}]$,
\begin{align*}
W_{k}(t) + \lVert D^kh(t) \rVert_{L^2} &\lesssim e^{Ct\lambda^{3-\sigma}\nu^{1/2}}t\lambda^{k+1-\sigma}(\lambda^{-1}\nu)^{2-\phi+\phi^k}(\lambda^{3-\sigma}\nu^{1/2}) \\
&\lesssim \lambda^{k+1-\sigma}(\lambda^{-1}\nu)^{2-\phi+\phi^k}.
\end{align*}
Therefore,
\begin{align*}
&\lVert D^k w(t) \rVert_{L^2} \lesssim \lambda^{k+1-\sigma}(\lambda^{-1}\nu)^{1-\phi+\phi^k} < \lambda^{k+1-\sigma}(\lambda^{-1}\nu)^{1-\phi+\phi^{k+1}}, \\
&\lVert D^k h(t) \rVert_{L^2} \lesssim \lambda^{k+1-\sigma}(\lambda^{-1}\nu)^{2-\phi+\phi^k} < \lambda^{k+1-\sigma}(\lambda^{-1}\nu)^{2-\phi+\phi^{k+1}}.
\end{align*}
\end{proof}
Since
\[
\lambda^{k+1-\sigma}(\lambda^{-1}\nu)^{1-\phi+\phi^{k+1}} \ll \lambda^{k+1-\sigma}(\lambda^{-1}\nu)^{1-\phi}
\]
by taking $\lambda$ large, Lemma \ref{lem-control-error} implies \eqref{boot-goal} on $[0,t_{0}]$. Thus the bootstrap is closed and \eqref{boot} holds on $[0,t_{*}]$. 
\begin{proposition}\label{prop-error}
If there exists a smooth solution to \eqref{pde:axisHallMHD} on $[0,t_{*}]$, then for $0 \le k \le 10$ and $t \in [0,t_{*}]$,
\[
\lVert w (t) \rVert_{\dot{H}^k} \le 2\lambda^{k+1-\sigma}(\lambda^{-1}\nu)^{1-\phi}, \quad \lVert h (t) \rVert_{\dot{H}^k} \le 2\lambda^{k+1-\sigma}(\lambda^{-1}\nu)^{2-\phi}.
\]
Since $B^\theta - \bar{B}^\theta = rh$ and $h$ is supported near $r\sim\nu^{-1}$, we also have
\[
\lVert B^\theta (t) e_\theta - \bar{B}^\theta(t)e_\theta \rVert_{\dot{H}^k} \simeq\lVert B^\theta (t) - \bar{B}^\theta(t) \rVert_{\dot{H}^k} \lesssim  \lambda^{k+1-\sigma}(\lambda^{-1}\nu)^{2-\phi}\nu^{-1}.
\]
\end{proposition}

\subsection{Proof of Theorem \ref{thm:axis}}\label{subsec-pf-axis}
\begin{proof}
For any $1<\sigma<7/2$, $\epsilon>0$, and $T_{*}>0$, we choose an initial data \eqref{ini_data}. By Lemma \ref{lem:ini_data},
\[
\lVert u_{0} \rVert_{H^{\sigma-1}} + \lVert B_{0} \rVert_{H^\sigma} \lesssim \epsilon. 
\]
Proposition \ref{prop-LWP} implies the local (in time) existence and uniqueness of smooth axisymmetric solution $(u,B)$. If its maximal life time is $T_{\max}$ and $T_{\max}<t_*$, then repeat what we did in Section \ref{sec-axi-error} replacing $t_{*}$ by $T_{\max}$ (note we do not have inflation at $T_{\max}<t_*$). Then \eqref{eq:ua-est} holds on $[0,T_{\max})$ since the bootstrap \eqref{boot} is closed on $[0,T_{\max}-\varepsilon]$ for any $\varepsilon>0$. So, recall $\|B(t)\|_{\dot{H}^k} \simeq \|ra(t)\|_{\dot{H}^k}$ thus $\|u(t)\|_{H^{10}}+\|B(t)\|_{H^{10}} < M_{\epsilon,\lambda} < \infty$ on $[0,T_{\max})$, where $M_{\epsilon,\lambda}$ does not depend on the initial time. Thus for any $0<s<T_{\max}$, by taking $m'=10$ in Proposition \ref{prop-LWP}, we could start at $(u(s),B(s))$ and establish existence on $[s,s+\delta]$ where $\delta = \delta(M_{\epsilon,\lambda})$ does not depend on $T_{\max}$. Taking $s$ close to $T_{\max}$, we get a contradiction to the maximality of $T_{\max}$.

Construct an approximate solution $(\bar{u},\bar{B})=(u_{0},\bar{B}^\theta e_{\theta})$ as in \eqref{pde:approx}. Lemma \ref{lem:Bbar} yields when $t = t_* = \epsilon^{-2}\lambda^{\sigma-3}\nu^{-1/2}$
\[
\lVert \bar B (t_{*}) \rVert_{\dot{H}^\sigma} \simeq \lVert \bar{B}^\theta(t_{*})e_\theta\rVert_{\dot{H}^\sigma} \simeq  \epsilon^{1-\sigma}.
\]
Note $t_{*}<T_*$ by choosing $\lambda$ sufficiently large since $\sigma < 7/2$. Lastly, Proposition \ref{prop-error} under interpolation implies that given a large $\lambda > 0$, $\forall t \in [0,t_{*}]$
\begin{align*}
\lVert u (t) \rVert_{H^{\sigma-1}} \lesssim& \lVert u (t) \rVert_{L^2} + \lVert u (t) \rVert_{\dot{H}^{\sigma-1}} \\
\le& \lVert u_{0} \rVert_{L^2} + \lVert w (t)\rVert_{L^2} + \lVert u_{0} \rVert_{\dot{H}^{\sigma-1}} + \lVert w (t)\rVert_{\dot{H}^{\sigma-1}} \lesssim \epsilon,
\end{align*}
and
\begin{align*}
\lVert B (t_{*})\rVert_{\dot{H}^{\sigma}}&=\lVert B^\theta (t_{*})e_\theta\rVert_{\dot{H}^{\sigma}} \\
&\ge \lVert \bar{B}^\theta(t_{*})e_\theta\rVert_{\dot{H}^\sigma} - \lVert B^\theta(t_{*})e_\theta - \bar{B}^\theta(t_{*})e_\theta \rVert_{\dot{H}^\sigma} \\
&\gtrsim \epsilon^{1-\sigma} - (\lambda^{-1}\nu)^{1-\phi} \gtrsim  \epsilon^{1-\sigma}.
\end{align*}
\end{proof}

\section{Inflation under Fractional Dissipation}\label{sec-frac}
In this section, we generalize our inflation to the case with fractional dissipation \eqref{pde:axisHallMHD-frac} and prove Theorem \ref{thm:frac}. We first make some changes to our setup in Section \ref{sec-data-app}. Denote
\[
\gamma := \max \{\alpha,\beta\} \in [0,5/4)
\]
and change our parameters \eqref{eq:nu-lam} slightly
\begin{equation}\label{eq:new-nu-lam}
\nu = \lambda^{1- \frac{7/2-2\gamma - \sigma}{100}} \ll \lambda,
\end{equation}
where we used $\sigma < 7/2 - 2\max\{\alpha,\beta\}$. All the results about initial data and approximate solution in Section \ref{sec-data-app} hold with our new parameters \eqref{eq:new-nu-lam}.

\subsection{Control of perturbation in mixed dimensions}
Next we control the perturbation. A very different point at our stage compared to inviscid case in Section \ref{sec-axi-error} is that dissipation destroys bounded support property in Lemma \ref{lem-bdd-supp}. So, we could no longer treat $\lVert B^\theta \rVert_{\dot{H}^k}$ and $\lVert B^\theta e_{\theta} \rVert_{\dot{H}^k}$ equivalently if $\mathrm{supp } \ B^\theta$ is not away from $z$-axis due to $|\nabla e_{\theta}| = 1/r$. So, we have to estimate vector-valued $H := B - \bar B = (B^\theta - \bar B^\theta) e_\theta$ instead of the scalar component $H^\theta=B^\theta - \bar B^\theta$. Luckily, we could make use of the 5D formulation of the system in Section \ref{subsec-5d}.

We still first assume the existence of smooth solution $(u,B)$ on $[0,t_{*}]$. Denote $a = \frac{B^\theta}{r}$, $\bar{a} = \frac{\bar{B}^\theta}{r}$, $h = a - \bar a$, and $w = u - \bar u$. Our goal is to control $\|w\|_{\dot{H}^k}$ and $\|H\|_{\dot{H}^{k}} = \|rhe_{\theta}\|_{\dot{H}^k}$. Consider the difference between \eqref{pde:ua-frac} and \eqref{pde:ua-bar}
\begin{equation}\label{pde:wh-frac}
\begin{split}
&\partial_t w +\bar u\cdot\nabla w +w\cdot\nabla\bar u + w\cdot\nabla w +\nabla q +rh(h+2\bar a)e_r + \eta (-\Delta_{\mathbb{R}^3})^\beta w = -\bar E_u - \eta (-\Delta_{\mathbb{R}^3})^\beta \bar{u}, \\
&\partial_{t}h + (\bar{u}-2\bar{a}e_{z})\cdot \nabla h  + (w-2he_{z})\cdot \nabla \bar{a} + (w-2he_{z})\cdot \nabla h + \mu \mathcal G_{5}^\alpha h = -\frac{\bar{E}_{B}^\theta}{r} -\mu \mathcal G_{5}^\alpha \bar{a}, \\
&\nabla \cdot w = 0.
\end{split}
\end{equation}
For a small times $t_{0}>0$, small $\phi>0$, and $0 \le k \le 10$,
\begin{equation}\label{boot-frac}
\lVert w(t) \rVert_{\dot{H}^k} \le 2 \lambda^{k+1-\sigma}(\lambda^{-1}\nu)^{1-\phi}, \quad \lVert \mathcal{P} h(t) \rVert_{\dot{H}^k(\mathbb{R}^5)} \le 2 \lambda^{k-\sigma}(\lambda^{-1}\nu)^{1-\phi}, \quad \forall t \in [0,t_{0}].
\end{equation}
Our goal is to show
\begin{equation}\label{boot-frac-goal}
\lVert w(t) \rVert_{\dot{H}^k} \le \lambda^{k+1-\sigma}(\lambda^{-1}\nu)^{1-\phi}, \quad \lVert \mathcal{P} h(t) \rVert_{\dot{H}^k(\mathbb{R}^5)} \le \lambda^{k-\sigma}(\lambda^{-1}\nu)^{1-\phi}, \quad \forall t \in [0,t_{0}].
\end{equation}
Then \eqref{boot-frac} holds on $[0,t_{*}]$ by a continuity argument. Since $\bar B$ is supported from $z$-axis, from Lemma \ref{lem:Bbar} and \ref{lem-Hk-3-5} for $t \in [0,t_*]$
\begin{equation}\label{eq:abar-L2est-frac}
\lVert \mathcal{P} \bar{a}(t) \rVert_{\dot{H}^k(\mathbb{R}^5)} \lesssim \lVert r\bar{a}(t)e_{\theta} \rVert_{\dot{H}^k} = \lVert \bar{B}^\theta e_{\theta} \rVert_{\dot{H}^k} \lesssim \lambda^{k-\sigma},
\end{equation}
and with \eqref{boot-frac}, $\bar u = u_0$, and Lemma \ref{lem:ini_data}
\begin{equation}\label{eq:ua-est-frac}
\begin{split}
&\lVert u (t) \rVert_{\dot{H}^k} \le \lVert w (t) \rVert_{H^k} + \lVert u_{0} (t) \rVert_{\dot{H}^k} \lesssim \lambda^{k+1-\sigma}, \\
&\lVert \mathcal{P} a(t) \rVert_{\dot{H}^k(\mathbb{R}^5)} \le \lVert \mathcal{P} h (t)\rVert_{H^k(\mathbb{R}^5)} + \lVert \mathcal{P} \bar{a}(t) \rVert_{H^k(\mathbb{R}^5)} \lesssim \lambda^{k-\sigma}.
\end{split}
\end{equation}
Also, Lemma \ref{lem:Bbar} and \ref{lem-Looo-3-5} imply
\begin{equation}\label{eq:abar-Looest-frac}
\lVert \mathcal{P} \bar{a} \rVert_{\dot{W}^{k,\infty}(\mathbb{R}^5)} \simeq \lVert \bar{a} \rVert_{\dot{W}^{k,\infty}} \simeq \nu \lVert \bar{B}^\theta \rVert_{\dot{W}^{k,\infty}}  \lesssim  \lambda^{k+1-\sigma}\nu^{3/2}.
\end{equation}

\begin{lemma}\label{lem-control-error-frac}
Under the bootstrap assumption \eqref{boot-frac}, for $0 \le k \le 10$,
\[
\lVert w \rVert_{\dot{H}^k} \lesssim \lambda^{k+1-\sigma} (\lambda^{-1}\nu)^{1-\phi+\phi^{k+1}}, \quad \lVert \mathcal{P}h \rVert_{\dot{H}^k(\mathbb{R}^5)} \lesssim \lambda^{k-\sigma} (\lambda^{-1}\nu)^{1-\phi+\phi^{k+1}}, \quad \forall t \in [0,t_{0}].
\]
\end{lemma}
\begin{proof}
For $0 \le m \le 8$, $0 \le l \le 7$ and small $0<\delta<\left( \frac{1}{2}-\phi \right)\frac{7/2-2\gamma-\sigma}{100}$, by Sobolev embedding and bootstrap assumption \eqref{boot-frac}
\begin{equation}\label{eq:whDooo-frac}
\begin{split}
&\lVert w \rVert_{\dot{W}^{m,\infty}} \lesssim \lVert w \rVert_{H^{m+3/2+\delta}} \lesssim \lambda^{m+5/2+\delta-\sigma}(\lambda^{-1}\nu)^{1-\phi} \lesssim \lambda^{m+2-\sigma}\nu^{1/2},\\
&\lVert \mathcal{P} h \rVert_{\dot{W}^{l,\infty}(\mathbb{R}^5)} \lesssim \lVert \mathcal{P} h \rVert_{H^{l+5/2+\delta}(\mathbb{R}^5)} \lesssim \lambda^{l+5/2+\delta-\sigma}(\lambda^{-1}\nu)^{1-\phi} \lesssim \lambda^{l+2-\sigma}\nu^{1/2}.
\end{split}
\end{equation}
Like in Lemma \ref{lem-control-error}, we will prove by induction, starting from $L^2$ case and inductively proving $\dot H^k$ case. We may omit the same steps but address the difference.

\noindent $\mathbf{L^2}$ \textbf{case:}
We multiply and integrate the $w$-equation in \eqref{pde:wh-frac} as in the inviscid case, and we now have two more dissipation terms.
\begin{align*}
&\partial_{t}\lVert w \rVert_{L^2}^2 + \eta\lVert (-\Delta_{\mathbb{R}^3})^{\beta/2}w\rVert_{L^2}^2 \\ 
\leq& (2\lVert \nabla \bar{u} \rVert_{L^\infty}+ \lVert \nabla w \rVert_{L^\infty})\lVert w \rVert_{L^2}^2 + \lVert r(h+2\bar{a}) h e_r\rVert_{L^2}\lVert w \rVert_{L^2} \\
&+ \lVert \bar{E}_{u} \rVert_{L^2}\lVert w \rVert_{L^2} + \eta\lVert (-\Delta_{\mathbb{R}^3})^{\beta}\bar{u} \rVert_{L^2}\lVert w \rVert_{L^2}.
\end{align*}
Drop the dissipation on the left and divide by $\|w\|_{L^2}$
\begin{align*}
\partial_{t}\lVert w \rVert_{L^2} \leq (2\lVert \nabla \bar{u} \rVert_{L^\infty}+ \lVert \nabla w \rVert_{L^\infty})\lVert w \rVert_{L^2} + \lVert r(h+2\bar{a}) h e_r\rVert_{L^2} + \lVert \bar{E}_{u} \rVert_{L^2} + \eta\lVert (-\Delta_{\mathbb{R}^3})^{\beta}\bar{u} \rVert_{L^2}.
\end{align*}
Note from \eqref{eq:whDooo-frac} and \eqref{eq:abar-Looest-frac}
\begin{align*}
\lVert r(h+2\bar{a}) h e_r\rVert_{L^2} \simeq \lVert \mathcal{P} (h+2\bar{a})\mathcal{P}h \rVert_{L^2(\mathbb{R}^5)} \le& \lVert \mathcal{P} (h + 2 \bar{a}) \rVert_{L^\infty(\mathbb{R}^5)}  \lVert\mathcal{P} h  \rVert_{L^2(\mathbb{R}^5)} \\
\lesssim& \lambda^{2-\sigma}\nu^{1/2}\lVert \mathcal{P} h \rVert_{L^2(\mathbb{R}^5)},
\end{align*}
and from Lemma \ref{lem:ini_data}
\[
\lVert (-\Delta_{\mathbb{R}^3})^{\beta}\bar{u} \rVert_{L^2} = \lVert u_{0} \rVert_{\dot{H}^{2\beta}} \lesssim \lambda^{2\beta+1-\sigma}.
\]
Therefore, together with other estimates we did in the proof of Lemma \ref{lem-control-error}
\begin{equation}\label{eq:est-w-2-frac}
\partial_{t}\lVert w \rVert_{L^2} \lesssim \lambda^{3-\sigma}\nu^{1/2}\lVert w \rVert_{L^2} + \lambda^{2-\sigma}\nu^{1/2} \lVert \mathcal{P}h \rVert_{L^2(\mathbb{R}^5)} + \lambda^{1-\sigma}(\lambda^{-1}\nu)(\lambda^{3-\sigma}\nu^{1/2}) + \lambda^{2\beta+1-\sigma}.
\end{equation}
For $h$-equation, we apply $\mathcal{P}$ to the whole equation and move to 5D space. So $h$-equation turns to
\begin{equation}\label{pde-h-5d}
\begin{split}
\partial_{t}\mathcal{P}h + \mathcal{P}(\bar{u}-2\bar{a}e_{z})\cdot \nabla_{\mathbb{R}^5} \mathcal{P} h  + \mathcal{P}(w-2he_{z})\cdot \nabla_{\mathbb{R}^5}\mathcal{P} \bar{a}& \\ + \mathcal{P}(w-2he_{z})\cdot \nabla_{\mathbb{R}^5}\mathcal{P} h
+ \mu (-\Delta_{\mathbb{R}^5})^\alpha \mathcal{P}h &= -\mathcal{P}\frac{\bar{E}_{B}^\theta}{r} -\mu (-\Delta_{\mathbb{R}^5})^\alpha \mathcal{P} \bar{a}.
\end{split}
\end{equation}
We estimate its $L^2$ norm similarly (and drop the dissipation on the left as above).
\begin{equation*}
\begin{split}
\partial_{t}\lVert\mathcal{P} h \rVert_{L^2(\mathbb{R}^5)}  
\lesssim& \lVert \nabla_{\mathbb{R}^5} \mathcal{P} (\bar{u} - 2\bar{a}e_{z}) \rVert_{L^\infty(\mathbb{R}^5)}\lVert \mathcal{P} h \rVert_{L^2(\mathbb{R}^5)} + \lVert \nabla_{\mathbb{R}^5} \mathcal{P} \bar{a} \cdot \mathcal{P}w \rVert_{L^2(\mathbb{R}^5)}  \\
&+ 2 \lVert \nabla_{\mathbb{R}^5}\mathcal{P} \bar{a} \rVert_{L^\infty(\mathbb{R}^5)} \lVert \mathcal{P}h \rVert_{L^2(\mathbb{R}^5)}
+ \lVert \nabla_{\mathbb{R}^5}\mathcal{P} (w-2he_{z}) \rVert_{L^\infty(\mathbb{R}^5)}\lVert \mathcal{P}h \rVert_{L^2(\mathbb{R}^5)}\\ 
&+ \left\lVert  \mathcal{P}\frac{\bar{E}_{B}^\theta}{r}  \right\rVert_{L^2(\mathbb{R}^5)}  +   \mu\lVert (-\Delta_{\mathbb{R}^5})^{\alpha} \mathcal{P}\bar{a} \rVert_{L^2(\mathbb{R}^5)} .
\end{split}
\end{equation*}
By Lemma \ref{lem:ini_data}, \ref{lem-Looo-3-5}, \eqref{eq:abar-Looest-frac} and \eqref{eq:whDooo-frac}
\begin{align*}
&\lVert \nabla_{\mathbb{R}^5} \mathcal{P} \bar{u} \rVert_{L^\infty(\mathbb{R}^5)}, \lVert \nabla_{\mathbb{R}^5} \mathcal{P} w \rVert_{L^\infty(\mathbb{R}^5)}, \lVert \nabla_{\mathbb{R}^5} \mathcal{P} h \rVert_{L^\infty(\mathbb{R}^5)} \lesssim \lambda^{3-\sigma}\nu^{1/2}, \\
&\lVert \nabla_{\mathbb{R}^5} \mathcal{P} \bar{a} \rVert_{L^\infty(\mathbb{R}^5)}\lesssim \lambda^{2-\sigma}\nu^{3/2}.
\end{align*}
Also, by Lemma \ref{lem:est Ebar}, \ref{lem-Hk-3-5}, and \eqref{eq:abar-L2est-frac}
\begin{align*}
&\left\lVert  \mathcal{P}\frac{\bar{E}_{B}^\theta}{r}  \right\rVert_{L^2(\mathbb{R}^5)} \simeq \lVert \bar{E}_{B}^\theta e_{\theta} \rVert_{L^2}  \lesssim \lambda^{-\sigma}(\lambda^{-1}\nu)(\lambda^{3-\sigma}\nu^{1/2}),\\
&\lVert (-\Delta_{\mathbb{R}^5})^{\alpha} \mathcal{P}\bar{a} \rVert_{L^2(\mathbb{R}^5)} \lesssim \lambda^{2\alpha-\sigma}.
\end{align*}
We pay special attention to $\lVert \nabla_{\mathbb{R}^5} \mathcal{P} \bar{a} \cdot \mathcal{P}w \rVert_{L^2(\mathbb{R}^5)}$ and be careful before applying H\"older inequality, because by Lemma \ref{lem-Hk-3-5} 
\[
\lVert \mathcal{P}w \rVert_{L^2(\mathbb{R}^5)} \simeq \lVert r w \rVert_{L^2} 
\]
instead of $\lVert w \rVert_{L^2}$. In fact,
\[
\lVert \nabla_{\mathbb{R}^5} \mathcal{P} \bar{a} \cdot \mathcal{P}w \rVert_{L^2(\mathbb{R}^5)} = \left\lVert  r\nabla_{\mathbb{R}^5} \mathcal{P} \bar{a} \cdot \frac{1}{r}\mathcal{P}w  \right\rVert_{L^2(\mathbb{R}^5)} \lesssim \lVert r\nabla\mathcal{P} \bar{a} \rVert_{L^\infty(\mathbb{R}^5)} \lVert w \rVert_{L^2} .
\]
Since $\bar a$ is supported near $r \sim \nu^{-1}$,
\[
\lVert r\nabla\mathcal{P} \bar{a} \rVert_{L^\infty(\mathbb{R}^5)} \lesssim \nu^{-1}\lVert \nabla\mathcal{P} \bar{a} \rVert_{L^\infty(\mathbb{R}^5)} \lesssim \lambda^{2-\sigma}\nu^{1/2}.
\]
Collecting these, we have
\begin{equation}\label{eq:est-h-2-frac}
\partial_{t}\lVert \mathcal{P} h \rVert_{L^2(\mathbb{R}^5)} \lesssim \lambda^{2-\sigma}\nu^{1/2}\lVert w \rVert_{L^2} + \lambda^{3-\sigma}\nu^{1/2} \lVert \mathcal{P}h \rVert_{L^2(\mathbb{R}^5)} + \lambda^{-\sigma}(\lambda^{-1}\nu)(\lambda^{3-\sigma}\nu^{1/2}) + \lambda^{2\alpha-\sigma}.
\end{equation}
Set $W(t) = \lambda^{-1}\|w(t)\|_{L^2}$. Multiply \eqref{eq:est-w-2-frac} by $\lambda^{-1}$ and add \eqref{eq:est-h-2-frac},
\begin{align*}
\partial_{t} (W + \lVert \mathcal{P}h \rVert_{L^2(\mathbb{R}^5)} ) \lesssim \lambda^{3-\sigma}\nu^{1/2} W + \lambda^{3-\sigma}\nu^{1/2}\lVert\mathcal{P} h \rVert_{L^2(\mathbb{R}^5)} + \lambda^{-\sigma}(\lambda^{-1}\nu)(\lambda^{3-\sigma}\nu^{1/2}) \\
+ \lambda^{2\beta-\sigma} + \lambda^{2\alpha-\sigma}.
\end{align*}
By Gr\"onwall inequality, on $[0,t_{0}]$ with $t_{0}\le t_* \lesssim \lambda^{\sigma-3}\nu^{-1/2}$
\begin{align*}
W(t) + \lVert \mathcal{P}h(t) \rVert_{L^2(\mathbb{R}^5)} &\lesssim e^{Ct\lambda^{3-\sigma}\nu^{1/2}}t(\lambda^{-\sigma}(\lambda^{-1}\nu)(\lambda^{3-\sigma}\nu^{1/2})+\lambda^{2\beta-\sigma} + \lambda^{2\alpha-\sigma}) \\
&\lesssim \lambda^{-\sigma}(\lambda^{-1}\nu).
\end{align*}
where the dissipation terms $\lambda^{2\beta-\sigma}$ and $\lambda^{2\alpha-\sigma}$ are absorbed by $\lambda^{-\sigma}(\lambda^{-1}\nu)(\lambda^{3-\sigma}\nu^{1/2}) = \lambda^{2-2\sigma}\nu^{3/2}$ due to \eqref{eq:new-nu-lam} and $2\gamma : = 2\max\{\alpha,\beta\} < 7/2 -\sigma$:
\[
\frac{\lambda^{2\gamma-\sigma}}{\lambda^{2-2\sigma}\nu^{3/2}} = \lambda^{2\gamma +\sigma-2}\nu^{-3/2} = \lambda^{2\gamma + \sigma - 7/2} \lambda^{\frac{3}{2}\cdot \frac{7/2- \sigma-2\gamma}{100}} \ll 1.
\]
So, we proved the $k=0$ case of this lemma.
\[
\lVert w(t) \rVert_{L^2} \lesssim \lambda^{1-\sigma}(\lambda^{-1}\nu), \quad \lVert \mathcal{P} h(t) \rVert_{L^2(\mathbb{R}^5)} \lesssim  \lambda^{-\sigma}(\lambda^{-1}\nu).
\]

\noindent $\mathbf{\dot{H}^k}$ \textbf{case:} We assume our Lemma \ref{lem-control-error-frac} holds for $m \le k-1$ and prove by induction. For the $u$-equation, we point out that the estimate is the same as in the proof of $\dot{H}^k$ case in Lemma \ref{lem-control-error}, with only a few exceptions: First we need to drop the dissipation on the left as above. On the right then we should care the two changed or new terms
\[
\lVert r(h+2\bar{a}) h e_r\rVert_{\dot{H}^k} \quad \text{and} \quad \eta\lVert (-\Delta_{\mathbb{R}^3})^{\beta}\bar{u} \rVert_{\dot{H}^k}
\]
From Lemma \ref{lem-Hk-3-5}
\begin{align*}
\lVert r(h+2\bar{a}) h e_r\rVert_{\dot{H}^k} =& \frac{1}{\sqrt{\pi}}\lVert \mathcal{P}(h+2\bar{a}) \mathcal{P}h\rVert_{\dot{H}^k(\mathbb{R}^5)} \\
\lesssim& \lVert \mathcal{P}h \mathcal{P}h\rVert_{\dot{H}^k(\mathbb{R}^5)} + \lVert \mathcal{P}\bar{a} \mathcal{P}h\rVert_{\dot{H}^k(\mathbb{R}^5)} \\
\lesssim& \sum_{i=0}^{\min \{k, 7\}} \lVert \mathcal{P}h \rVert_{\dot{W}^{i,\infty}(\mathbb{R}^5)} \lVert \mathcal{P}h \rVert_{\dot{H}^{k-i}(\mathbb{R}^5)}   \\
&+ \sum_{i=\min \{k+1, 8\}}^{k} \lVert \mathcal{P}h \rVert_{\dot{H}^i(\mathbb{R}^5)} \lVert \mathcal{P}h \rVert_{\dot{W}^{k-i,\infty}(\mathbb{R}^5)} \\
&+ \sum_{i=0}^k \lVert \mathcal{P}h \rVert_{\dot{H}^i(\mathbb{R}^5)} \lVert \mathcal{P}\bar{a} \rVert_{\dot{W}^{k-i,\infty}(\mathbb{R}^5)}
\end{align*}
where the top order terms are estimated using \eqref{eq:whDooo-frac} and \eqref{eq:abar-Looest-frac}
\[
\lVert \mathcal{P}h \rVert_{\dot{H}^k(\mathbb{R}^5)} (\lVert \mathcal{P}h \rVert_{L^{\infty}(\mathbb{R}^5)}+ \lVert \mathcal{P}\bar{a} \rVert_{L^{\infty}(\mathbb{R}^5)} ) \lesssim \lambda^{2-\sigma}\nu^{1/2}\lVert \mathcal{P}h \rVert_{\dot{H}^k(\mathbb{R}^5)}.
\]
The lower terms are dominated by $\lambda^{k+1-\sigma}(\lambda^{-1}\nu)^{1-\phi+\phi^{k}}(\lambda^{3-\sigma}\nu^{1/2})$.
\begin{align*}
&\sum_{i=1}^{\min \{k, 7\}} \lVert \mathcal{P}h \rVert_{\dot{W}^{i,\infty}(\mathbb{R}^5)} \lVert \mathcal{P}h \rVert_{\dot{H}^{k-i}(\mathbb{R}^5)} + \sum_{i=\min \{k+1, 8\}}^{k-1} \lVert \mathcal{P}h \rVert_{\dot{H}^i(\mathbb{R}^5)} \lVert \mathcal{P}h \rVert_{\dot{W}^{k-i,\infty}(\mathbb{R}^5)} \\
&+ \sum_{i=0}^{k-1} \lVert \mathcal{P}h \rVert_{\dot{H}^i(\mathbb{R}^5)} \lVert \mathcal{P}\bar{a} \rVert_{\dot{W}^{k-i,\infty}(\mathbb{R}^5)}\\
\lesssim& \sum_{i=1}^{\min \{k, 7\}} \lambda^{i+2-\sigma}\nu^{1/2}\lambda^{k-i-\sigma} (\lambda^{-1}\nu)^{1-\phi+\phi^{k-i+1}} + \sum_{i=\min \{k+1, 8\}}^{k-1}\lambda^{i-\sigma} (\lambda^{-1}\nu)^{1-\phi+\phi^{i+1}}\lambda^{k-i+2-\sigma}\nu^{1/2}\\
&+\sum_{i=0}^{k-1} \lambda^{i-\sigma} (\lambda^{-1}\nu)^{1-\phi+\phi^{i+1}}\lambda^{k-i+1-\sigma}\nu^{3/2} \\
\lesssim& \lambda^{k+1-\sigma}(\lambda^{-1}\nu)^{1-\phi+\phi^{k}}(\lambda^{3-\sigma}\nu^{1/2})\lambda^{-2}.
\end{align*}
Also from Lemma \ref{lem:ini_data}
\[
\lVert (-\Delta_{\mathbb{R}^3})^{\beta}\bar{u} \rVert_{\dot{H}^k} \lesssim \lambda^{2\beta + k + 1 - \sigma}.
\]
Then, for the other terms using the same estimate as in the proof of Lemma \ref{lem-control-error}
\begin{equation}\label{eq:w-k-frac}
\begin{split}
\partial_{t}\lVert w \rVert_{\dot{H}^k} \lesssim& \lambda^{3-\sigma}\nu^{1/2} \lVert w \rVert_{\dot{H}^k} + \lambda^{2-\sigma}\nu^{1/2}\lVert \mathcal{P}h \rVert_{\dot{H}^k(\mathbb{R}^5)} + \lambda^{k+1-\sigma}(\lambda^{-1}\nu)^{1-\phi+\phi^{k}}(\lambda^{3-\sigma}\nu^{1/2})\\
&+ \lambda^{k+1-\sigma}(\lambda^{-1}\nu)(\lambda^{3-\sigma}\nu^{1/2}) + \lambda^{2\beta+k+1-\sigma}.
\end{split}
\end{equation}
For $h$-equation, we point it out that the nonlinear term could be split into $\dot{H}^i(\mathbb{R}^5)$ and $\dot{W}^{k-i,\infty}(\mathbb{R}^5)$ norms as in the proof of Lemma \ref{lem-control-error} except two terms that we need to be careful:
\[
\int_{\mathbb{R}^5} D_{\mathbb{R}^5}^k (\nabla_{\mathbb{R}^5} \mathcal{P}\bar{a} \cdot \mathcal{P} w) D_{\mathbb{R}^5}^k (\mathcal{P}h)dy, \quad \int_{\mathbb{R}^5} D_{\mathbb{R}^5}^k (\nabla_{\mathbb{R}^5} \mathcal{P}h \cdot \mathcal{P} w) D_{\mathbb{R}^5}^k (\mathcal{P}h)dy.
\]
The former could be estimated as in $L^2$ case since the estimate of $\bar a$ \eqref{eq:abar-Looest-frac} is known and $\bar a$ is compactly supported. Here we estimate the harder one, that is, the latter. Note
\[
\int_{\mathbb{R}^5} (\nabla_{\mathbb{R}^5} D_{\mathbb{R}^5}^k\mathcal{P}h \cdot \mathcal{P} w) D_{\mathbb{R}^5}^k (\mathcal{P}h)dy = - \frac{1}{2} \int_{\mathbb{R}^5} (\nabla_{\mathbb{R}^5} \cdot \mathcal{P}w) |D_{\mathbb{R}^5}^k (\mathcal{P}h)|^2 dy.
\]
We remark that on $\mathbb{R}^3$, $\nabla \cdot w = 0$, but on $\mathbb{R}^5$, $\nabla_{\mathbb{R}^5} \cdot \mathcal{P}w = \frac{2\mathcal{P}w^r}{r} \neq 0$. Thus
\begin{align*}
\left|\int_{\mathbb{R}^5} D_{\mathbb{R}^5}^k (\nabla_{\mathbb{R}^5} \mathcal{P}h \cdot \mathcal{P} w) D_{\mathbb{R}^5}^k (\mathcal{P}h)dy\right|\lesssim& \sum_{i=1}^{\min\{k,8\}} \lVert \mathcal{P}w \rVert_{\dot{W}^{i,\infty}(\mathbb{R}^5)}\lVert \mathcal{P}h \rVert_{\dot{H}^{k-i+1}(\mathbb{R}^5)}\lVert \mathcal{P}h \rVert_{\dot{H}^{k}(\mathbb{R}^5)} \\
&+ \lVert D_{\mathbb{R}^5}^{k-1}\mathcal{P}w \cdot D_{\mathbb{R}^5}^2\mathcal{P} h \rVert_{L^2(\mathbb{R}^5)}\lVert \mathcal{P}h \rVert_{\dot{H}^{k}(\mathbb{R}^5)}  \\
&+ \lVert D_{\mathbb{R}^5}^{k}\mathcal{P}w \cdot D_{\mathbb{R}^5}\mathcal{P} h \rVert_{L^2(\mathbb{R}^5)}\lVert \mathcal{P}h \rVert_{\dot{H}^{k}(\mathbb{R}^5)}.
\end{align*}
where the last two terms do not exist when $k \le 9$ or $k \le 8$ respectively. The first term could be estimated as in the proof of Lemma \ref{lem-control-error}, where $\|w\|_{\dot{W}^{m,\infty}} \simeq \|\mathcal{P}w\|_{\dot{W}^{m,\infty}(\mathbb{R}^5)}$ from Lemma \ref{lem-Looo-3-5}. But as in the $L^2$ case, $\|w\|_{\dot{H}^m}$ and $\|\mathcal{P}w\|_{\dot{H}^m(\mathbb{R}^5)}$ are not equivalent. So, if the last two exist, we cannot estimate them by H\"older inequality directly. Instead, we use Lemma \ref{lem-3-5-nonlinear}.
\begin{align*}
\lVert D_{\mathbb{R}^5}^{k}\mathcal{P}w \cdot D_{\mathbb{R}^5}\mathcal{P} h \rVert_{L^2(\mathbb{R}^5)} &\lesssim \lVert \nabla^k w \nabla H \rVert_{L^2}  \\
&\lesssim \lVert \nabla^k w \rVert_{L^2} \lVert \nabla H \rVert_{L^\infty} 
\end{align*}
and 
\begin{align*}
\lVert D_{\mathbb{R}^5}^{k-1}\mathcal{P}w \cdot D_{\mathbb{R}^5}^2\mathcal{P} h \rVert_{L^2(\mathbb{R}^5)} &\lesssim \lVert \nabla^{k-1} w \nabla^2 H \rVert_{L^2}  \\
&\lesssim \lVert \nabla^{k-1}w \rVert_{L^2} \lVert \nabla^2 H \rVert_{L^\infty} .
\end{align*}
By Sobolev embedding, for $m = 1$ or $2$ and small $\delta>0$
\[
\|\nabla^m H\|_{L^\infty} \lesssim \|H\|_{H^{m+3/2+\delta}} \lesssim \|\mathcal{P}h\|_{H^{m+3/2+\delta}(\mathbb{R}^5)} \lesssim \lambda^{m+3/2+\delta-\sigma}(\lambda^{-1}\nu)^{1-\phi} \lesssim \lambda^{m+1-\sigma}\nu^{1/2}.
\]
As a result, by induction on $\|w\|_{\dot{H}^m}$ for $m \le k-1$,
\begin{align*}
&\lVert D_{\mathbb{R}^5}^{k-1}\mathcal{P}w \cdot D_{\mathbb{R}^5}^2\mathcal{P} h \rVert_{L^2(\mathbb{R}^5)} + \lVert D_{\mathbb{R}^5}^{k}\mathcal{P}w \cdot D_{\mathbb{R}^5}\mathcal{P} h \rVert_{L^2(\mathbb{R}^5)} \\
\lesssim& \lVert w \rVert_{\dot{H}^{k-1}} \lVert H \rVert_{\dot{W}^{2,\infty}} + \lVert w \rVert_{\dot{H}^{k}} \lVert H \rVert_{\dot{W}^{1,\infty}} \\
\lesssim& \lambda^{k-\sigma}(\lambda^{-1}\nu)^{1-\phi+\phi^k}\lambda^{3-\sigma}\nu^{1/2} + \lambda^{2-\sigma}\nu^{1/2}\lVert w \rVert_{\dot{H}^k} .
\end{align*}
Lastly, using Lemma \ref{lem:est Ebar}, \ref{lem-Hk-3-5}, and \eqref{eq:abar-L2est-frac}
\begin{align*}
&\left\lVert  \mathcal{P}\frac{\bar{E}_{B}^\theta}{r}  \right\rVert_{\dot{H}^k(\mathbb{R}^5)} \lesssim \lambda^{k-\sigma}(\lambda^{-1}\nu)(\lambda^{3-\sigma}\nu^{1/2}),\\
&\lVert (-\Delta_{\mathbb{R}^5})^{\alpha} \mathcal{P}\bar{a} \rVert_{\dot{H}^k(\mathbb{R}^5)} \lesssim \lambda^{2\alpha+k-\sigma}.
\end{align*}
So, we have
\begin{equation}\label{eq:h-k-frac}
\begin{split}
\partial_{t}\lVert \mathcal{P} h \rVert_{\dot{H}^k(\mathbb{R}^5)} \lesssim& \lambda^{2-\sigma}\nu^{1/2}\lVert w \rVert_{\dot{H}^k}+ \lambda^{3-\sigma}\nu^{1/2} \lVert \mathcal{P}h \rVert_{\dot{H}^k(\mathbb{R}^5)} +  \lambda^{k-\sigma}(\lambda^{-1}\nu)^{1-\phi+\phi^{k}}(\lambda^{3-\sigma}\nu^{1/2}) \\
& + \lambda^{k-\sigma}(\lambda^{-1}\nu)(\lambda^{3-\sigma}\nu^{1/2}) + \lambda^{2\alpha+k-\sigma}.
\end{split}
\end{equation}
Set $W_{k}(t) = \lambda^{-1}\|w(t)\|_{\dot{H}^k}$. Add \eqref{eq:h-k-frac} and \eqref{eq:w-k-frac} multiplied by $\lambda^{-1}$
\begin{align*}
\partial_{t}(W_{k}+\lVert \mathcal{P} h \rVert_{\dot{H}^k(\mathbb{R}^5)}) \lesssim& \lambda^{3-\sigma}\nu^{1/2}W_{k} + \lambda^{3-\sigma}\nu^{1/2} \lVert \mathcal{P}h \rVert_{\dot{H}^k(\mathbb{R}^5)}\\
& +  \lambda^{k-\sigma}(\lambda^{-1}\nu)^{1-\phi+\phi^{k}}(\lambda^{3-\sigma}\nu^{1/2})
+ \lambda^{k-\sigma}(\lambda^{-1}\nu)(\lambda^{3-\sigma}\nu^{1/2})\\
&+ \lambda^{2\beta+k-\sigma} + \lambda^{2\alpha+k-\sigma}.
\end{align*}
Gr\"onwall inequality then implies that on $[0,t_{0}]$
\begin{align*}
W_{k}(t) + \lVert \mathcal{P}h(t) \rVert_{\dot{H}^k(\mathbb{R}^5)} &\lesssim e^{Ct\lambda^{3-\sigma}\nu^{1/2}}t(\lambda^{k-\sigma}(\lambda^{-1}\nu)^{1-\phi+\phi^{k}}(\lambda^{3-\sigma}\nu^{1/2})+\lambda^{2\beta+k-\sigma} + \lambda^{2\alpha+k-\sigma}) \\
&\lesssim \lambda^{k-\sigma}(\lambda^{-1}\nu)^{1-\phi+\phi^k} < \lambda^{k-\sigma}(\lambda^{-1}\nu)^{1-\phi+\phi^{k+1}}.
\end{align*}
where the dissipation terms $\lambda^{2\beta+k-\sigma}$ and $\lambda^{2\alpha+k-\sigma}$ are absorbed by $\lambda^{k-\sigma}(\lambda^{-1}\nu)^{1-\phi+\phi^k}(\lambda^{3-\sigma}\nu^{1/2}) = \lambda^{k+2-2\sigma} \nu^{3/2}(\lambda^{-1}\nu)^{-\phi+\phi^k}$: As in the $L^2$ case, from \eqref{eq:new-nu-lam} and $2\gamma : = 2\max\{\alpha,\beta\} < 7/2 -\sigma$:
\[
\frac{\lambda^{2\gamma+k-\sigma}}{\lambda^{k+2-2\sigma} \nu^{3/2}(\lambda^{-1}\nu)^{-\phi+\phi^k}} < \frac{\lambda^{2\gamma+k-\sigma}}{\lambda^{k+2-2\sigma} \nu^{3/2}} = \lambda^{2\gamma +\sigma-2}\nu^{-3/2} = \lambda^{2\gamma + \sigma - 7/2} \lambda^{\frac{3}{2}\cdot \frac{7/2- \sigma-2\gamma}{100}} \ll 1.
\]
Therefore, $\dot{H}^k$ case of the lemma is also proved.
\end{proof}
Note
\[
\lambda^{k-\sigma}(\lambda^{-1}\nu)^{1-\phi+\phi^{k+1}} \ll \lambda^{k-\sigma}(\lambda^{-1}\nu)^{1-\phi}.
\]
Thus \eqref{boot-frac-goal} holds on $[0,t_{0}]$ if $\lambda$ is large, and a continuity argument makes \eqref{boot-frac} holds on $[0,t_{*}]$. Also, from Lemma \ref{lem-Hk-3-5}, for $H = H^\theta e_\theta = rhe_\theta$
\[
\lVert H \rVert_{\dot{H}^k} \simeq \lVert \mathcal{P}h \rVert_{\dot{H}^k(\mathbb{R}^5)} \lesssim \lambda^{k-\sigma}(\lambda^{-1}\nu)^{1-\phi}.
\]
\begin{proposition}\label{prop-frac}
If there exists a smooth solution to \eqref{pde:axisHallMHD-frac} on $[0,t_{*}]$, then for $0 \le k \le 10$ and $t \in [0,t_{*}]$,
\[
\lVert w (t) \rVert_{\dot{H}^k} \le 2\lambda^{k+1-\sigma}(\lambda^{-1}\nu)^{1-\phi}, \quad \lVert H (t) \rVert_{\dot{H}^k} \lesssim \lambda^{k-\sigma}(\lambda^{-1}\nu)^{1-\phi}.
\]
\end{proposition}

\subsection{Proof of Theorem \ref{thm:frac}}
\begin{proof}[Proof of Theorem \ref{thm:frac}]
For any $1<\sigma<7/2-2\max \{\alpha,\beta\}$, $\epsilon>0$, and $T_{*}>0$, we choose an initial data \eqref{ini_data}. By Lemma \ref{lem:ini_data} (adopted with the new parameters \eqref{eq:new-nu-lam}),
\[
\lVert u_{0} \rVert_{H^{\sigma-1}} + \lVert B_{0} \rVert_{H^\sigma} \lesssim \epsilon. 
\]
This existence and uniqueness of smooth solution up to $[0,t_{*}]$ is proved by Proposition \ref{prop-LWP} and maximality contradiction as in Section \ref{subsec-pf-axis}. The only difference we should mention is that $\|B(t)\|_{\dot{H}^k} \simeq \|\mathcal{P} a(t)\|_{\dot{H}^k(\mathbb{R}^5)}$ thus \eqref{boot-frac} reads $\|u(t)\|_{H^{10}}+\|B(t)\|_{H^{10}} < M_{\epsilon,\lambda} < \infty$.

Construct an approximate solution $(\bar{u},\bar{B})=(u_{0},\bar{B}^\theta e_{\theta})$ as in \eqref{pde:approx} (still using the new parameters \eqref{eq:new-nu-lam}). Lemma \ref{lem:Bbar} yields when $t = t_* = \epsilon^{-2}\lambda^{\sigma-3}\nu^{-1/2}$
\[
\lVert \bar B (t_{*}) \rVert_{\dot{H}^\sigma} \simeq \lVert \bar{B}^\theta(t_{*})e_\theta\rVert_{\dot{H}^\sigma} \simeq  \epsilon^{1-\sigma}.
\]
Note $t_{*}<T_*$ by choosing $\lambda$ sufficiently large since $\sigma < 7/2$. Lastly, Proposition \ref{prop-frac} under interpolation implies that given a large $\lambda > 0$, $\forall t \in [0,t_{*}]$
\begin{align*}
\lVert u (t) \rVert_{H^{\sigma-1}} \lesssim& \lVert u (t) \rVert_{L^2} + \lVert u (t) \rVert_{\dot{H}^{\sigma-1}} \\
\le& \lVert u_{0} \rVert_{L^2} + \lVert w (t)\rVert_{L^2} + \lVert u_{0} \rVert_{\dot{H}^{\sigma-1}} + \lVert w (t)\rVert_{\dot{H}^{\sigma-1}} \lesssim \epsilon,
\end{align*}
and
\begin{align*}
\lVert B (t_{*})\rVert_{\dot{H}^{\sigma}}&\ge \lVert \bar{B}(t_{*})\rVert_{\dot{H}^\sigma} - \lVert H (t_{*})\rVert_{\dot{H}^\sigma} \\
&\gtrsim \epsilon^{1-\sigma} - (\lambda^{-1}\nu)^{1-\phi} \gtrsim  \epsilon^{1-\sigma}.
\end{align*}
\end{proof}

\section*{References}

\printbibliography[heading=none]

\end{document}